\pdfoutput=1
\documentclass[preprint,12pt]{elsarticle}

\usepackage{graphicx}

\usepackage{bm}
\usepackage{amsmath}
\usepackage{amssymb}
\usepackage{amsthm}
\usepackage{color}
\usepackage{ulem}
\usepackage{hyperref}

\newtheorem{theorem}{Theorem}
\newtheorem{lemma}{Lemma}

\newtheorem{remark}{Remark}

\newcommand{\IGNORE}[1]{}
\newcommand{\ignore}[1]{}

\newcommand{\mbb}[1]{\mathbb{#1}}
\newcommand{\mb}[1]{\mathbf{#1}}

\newcommand{\der}[2]{\frac{\partial #1}{\partial #2}}
\newcommand{\derr}[2]{\frac{\delta #1}{\delta #2}}
\newcommand{\dder}[2]{\frac{\partial^2 #1}{\partial #2^2}}

\newcommand{\x}{\tilde{x}}

\newcommand{\dx}{\; {\rm d} x}
\newcommand{\dy}{\; {\rm d} y}

\renewcommand{\div}{{\rm\bf div \;}}
\newcommand{\DIV}{{\cal DIV}}
\newcommand{\bDIV}{{\cal \bf DIV}}
\newcommand{\GRAD}{{\cal GRAD}}
\newcommand{\bGRAD}{{\cal \bf GRAD}}

\def\Frac{\displaystyle\frac}

\journal{Journal of Computational Physics}

\begin{document}

\begin{frontmatter}

\title{{\bf The mimetic finite difference method for the Landau-Lifshitz equation}}

\author[berk,lanl]{Eugenia Kim}
\author[lanl]{Konstantin Lipnikov}

\address[lanl]{MS-B284, Los Alamos National Laboratory, Los Alamos, NM 87544}
\address[berk]{Department of Mathematics, University of California,
Berkeley, CA 94720}
\begin{abstract}
The Landau-Lifshitz equation describes the dynamics of the magnetization inside ferromagnetic materials.
This equation is highly nonlinear and has a non-convex constraint 
(the magnitude of the magnetization is constant)
which pose interesting challenges in developing numerical methods.
We develop and analyze explicit and implicit mimetic finite difference schemes for this equation.
These schemes work on general polytopal meshes which provide enormous flexibility to 
model magnetic devices with various shapes.
A projection on the unit sphere is used to preserve the magnitude of the magnetization. 
We also provide a proof that shows the exchange energy is decreasing in certain conditions.
The developed schemes are tested on general meshes that include distorted and randomized meshes. 
The numerical experiments include a test proposed by the National Institute of Standard and Technology 
and a test showing formation of domain wall structures in a thin film.
\end{abstract}

\begin{keyword}
  micromagnetics\sep
  Landau-Lifshitz equation\sep 
  Landau-Lifshitz-Gilbert equation\sep 
  mimetic finite difference method \sep 
  polygonal meshes

\end{keyword}

\end{frontmatter}

%%%%%%%%%%%%%%%%%%%%%%%%%%%%%%%%%%%%%%%%%%%%%%%%%%%%%%%%%%%%%%%%%%%%%
\section{Introduction}
\setcounter{equation}{0}

Micromagnetics studies behavior of ferromagnetic materials at sub-micrometer 
length scales \cite{fidler2000micromagnetic}.
These scales are large enough to use a continuum PDE model and are small enough 
to resolve important magnetic structures such as domain walls, vortices and skyrmions \cite{lin2013manipulation}.
The dynamics of the magnetic distribution $\bm{m}$ in a ferromagnetic material is 
governed by the Landau-Lifshitz (LL) equation.
There exist several equivalent forms of the LL equation, such as the Landau-Lifshitz-Gilbert equation,
that lead to a large family of numerical methods.
%Nevertheless, to the best our knowledge, the schemes proposed in this paper are the first 
%ones that are based on a mixed formulation of the LL equation and work on arbitrary polygonal 
%(polyhedral in 3D) meshes.

The evolution of $\bm{m}$ is driven by the effective field $\mb{h}$ which can 
be described as the functional derivative of the LL energy density with respect to 
the magnetization.
The LL energy is given by
$$
  E(\mb{m})
   = \frac{\eta}{2} \int_V |\nabla \mb{m}|^2 \dx  
   + \frac{Q}{2} \int_V (m_2^2+ m^2_3)\dx 
   - \frac{1}{2}\int_V \mb{h_s}\cdot \mb{m} \dx
   - \int_V \mb{h_e} \cdot \mb{m} \dx.
$$
The first term is the exchange energy, which favors the alignment of the magnetization along a common direction. 
The second term is the anisotropy energy, which prefers the certain orientation of the spins due 
to the crystalline lattice. 
The third term is the stray field energy, which is induced by the magnetization distribution inside the material.
The last term is the external field energy which favors the orientation of the spins along an external field. 
The exchange energy, anisotropy energy and the external field energy are local terms in that a local change in the magnetization 
affects only locally, whereas the stray field energy is the nonlocal term in that a local change in the magnetization 
affects globally \cite{cimrak2007survey}.

The LL equation has a few important properties \cite{cimrak2007survey}.
First, the magnitude of the magnetization is preserved, namely $|\mb{m}|=M$.
We can renormalize the LL equation so that $M=1$.
Secondly, the energy decreases in time, in case of a constant applied field, which is called the Lyapunov structure.
Lastly, if there is no damping, i.e. $\alpha =0$, the energy is conserved, which is called the Hamiltonian structure.

Various numerical methods have been developed for the LL equation, see e.g. review papers 
\cite{cimrak2007survey,garcia2007numerical,kruzik2006recent}.
The discretization strategies in space are discussed in the following articles :
In \cite{miltat2007numerical}, the finite difference methods based on the field and energy are presented.
The field-based finite difference method is obtained by discretizing field $\mb{h}$ itself, 
whereas in the energy based approach this field is derived from the discretized energy. 
A finite element method is employed in \cite{fidler2000micromagnetic}.
The magnetization is approximated with piecewise linear functions and 
the effective field $\mb{h}$ is obtained as the first variation of the discretized energy.
In \cite{schrefl1999finite}, the finite element method with piecewise linear functions is 
applied to the Landau-Lifshitz-Gilbert equation, which is another formulation of the LL equation.

A large family of time stepping schemes have been developed 
which conserve the magnitude of the magnetization.
The Gauss-Seidel projection method developed in \cite{wang2000numerical,wang2001gauss, garcia2003improved} uses
another formulation of the LL equation (the last equation in \eqref{eq:LLsimplified})
and treats $|\nabla \bm{m}|^2$ as the Lagrange multiplier for the pointwise constraint $|\bm{m}|=1$.
The gyromagnetic and damping terms are treated separately to overcome the difficulties associated with 
the stiffness and nonlinearity.
The resulting method is first-order accurate and unconditionally stable.
In \cite{jiang2001hysteresis}, the semi-analytic integration method is developed by analytically integrating 
the system of ODEs appearing after a spatial discretization of the LL equation.
This method is the first-order accurate 
but explicit, hence is subject to a Courant time step constraint.
The geometric integration method has been applied in \cite{krishnaprasad2001cayley}, 
and in a more general setting in \cite{lewis2003geometric},
using the Cayley transform to lift the LL equation to the Lie algebra of the three dimensional rotation group.
Unlike the semi-analytic integration methods, this method is more amenable for building numerical schemes with higher-order accuracy.
Finally, we mention the method based on the mid-point rule \cite{bertotti2001nonlinear,d2005geometrical} which 
is second-order accurate, unconditionally stable, and preserves the magnitude of the magnetization, 
as well as the Lyapunov and Hamiltonian structures of the LL equation.
Although these methods could be extended to finite element discretizations, 
to the best of our knowledge, the literature has only examples of finite difference schemes,
which are difficult to use for general domains.

The semi-discrete schemes are introduced in \cite{prohl2001computational} for 2D and 
in \cite{cimrak2005error} for 3D formulation of the LL equation and error estimates 
are derived under the assumption that there exist a strong solution.

The finite element methods for the LL equation are {typically} presented with rigorous convergence 
analysis that deals with weak solutions.
In \cite{alouges2006convergence,alouges2008new,alouges2012convergent}, the finite element method is 
developed for an equivalent formulation of the LL equation (see, formula \eqref{eq:LLG2}) which is the 
first-order accurate (in the energy norm) in both space and time and requires only 
one linear solver on each time step. 
In \cite{kritsikis2014beyond,alouges2014convergent}, the method is developed further to achieve the 
second-order accuracy in time.
In \cite{bartels2006convergence}, Bartels and Prohl considered an implicit time integration method
for the Landau-Lifshitz-Gilbert equation (see, formula \eqref{eq:LLG}), 
which is unconditionally stable, but a nonlinear solver is needed on each time step.
In \cite{cimrak2009convergence}, Cimr{\'a}k proposed a scheme for the LL equation, 
using a midpoint rule that could be easily adapted to the limiting cases,
but a nonlinear solver is needed on each time step. 

We present explicit and implicit mimetic finite difference (MFD) schemes \cite{MFD-book} 
for the LL equation.
In contrast to the existing numerical methods that use conventional spatial discretizations with various time 
stepping strategies, we deliver a new spatial discretization which has a number of unique properties.
First, it works on arbitrary polytopal (polygonal in 2D and polyhedral in 3D) meshes including locally refined 
meshes with degenerate cells.
For the same mesh resolution, polytopal meshes need fewer cells to cover the domain than simplicial meshes, which leads 
to fewer number of unknowns and a more efficient scheme.
Elegant treatment of degenerate cells that appear in adaptive mesh refinement/coarsening algorithms allows us to track 
accurate dynamics of domain walls.
Although mesh adaptation is beyond the focus of this paper, we illustrate the underlying idea with one numerical experiment.
Secondly, we use a mixed formulation of the LL equation that simplifies numerical control of the constraint $|\mb{m}|=1$.
To the best of our knowledge, the schemes proposed in this paper are the first 
ones that are based on a mixed formulation of the LL equation. 
Thirdly, the MFD could be applied to problems posed in general domains, like the finite element methods,
which is a key advantage compared to the finite difference methods. 
To the best of our knowledge, the time stepping schemes such as GSPM \cite{wang2000numerical,wang2001gauss, garcia2003improved}
 and other geometric methods \cite{jiang2001hysteresis,krishnaprasad2001cayley, lewis2003geometric, bertotti2001nonlinear,d2005geometrical}
have only been tested on finite difference stencils.
Fourthly, we prove that the exchange energy decreases on polygonal meshes under certain conditions,
which were not addressed in GSPM \cite{wang2000numerical,wang2001gauss, garcia2003improved}
and other methods \cite{jiang2001hysteresis,krishnaprasad2001cayley, lewis2003geometric, bertotti2001nonlinear}.
Fifthly, our implicit scheme has the similar complexity as the algorithms in Alouges  \cite{alouges2006convergence,alouges2008new,alouges2012convergent}, 
in that we need to solve a linear system for each time step.
Finally, compared to the methods in \cite{alouges2006convergence,alouges2008new,alouges2012convergent,bartels2006convergence},
this method is developed for the LL equation, which makes it more suitable to apply to the limiting cases, 
which is important for physical simulations.

The MFD method was originally designed to preserve or mimic important mathematical and physical
properties of continuum PDEs in discrete schemes on unstructured polytopal meshes.
It has been successfully employed for solving diffusion, convection-diffusion, electromagnetic, 
and linear elasticity problems and for modeling various fluid flows.
The original MFD method is a low-order method, but miscellaneous
approaches were developed towards higher-order methods.
We refer to book \cite{MFD-book} and review paper \cite{Lipnikov-Manzini-Shashkov:2014} 
for extensive review of mimetic schemes.
This the first application of the mimetic discretization technology
to a geometric dispersive partial differential equation. 

The mimetic discretization framework combines rich tools of a finite element analysis with 
the flexibility of finite volume meshes. 
Here, we consider explicit and implicit time integration schemes and
demonstrate the flexibility of this framework with various numerical experiments.
We perform stability analysis on polygonal meshes but defer rigorous convergence analysis for future work.
%In the proposed numerical algorithm, at most one linear system has to be solved on each time step.

The computation of the stray field term is the most time consuming part of each micromagnetic simulation
due to its nonlocal nature. 
Numerous numerical methods for the stray field calculation are described and studied in \cite{abert2013numerical}. 
They can be divided into two groups.
The first group includes methods that solve a PDE in $\mathbb{R}^d$ for the potential field.
Since this PDE is posed in an infinite domain, hybrid numerical methods are typically used.
The finite element and boundary element methods are used in \cite{fredkin1990hybrid, garcia2006adaptive}, 
the finite element method and the shell transformation are used in \cite{brunotte1992finite}.
The computation cost is reduced by using multigrid preconditioners \cite{tsukerman1998multigrid} and
$\mathcal{H}$-matrix approximation \cite{popovic2005applications}.
The second group includes methods based on direct evaluation of the integral with a nonlocal kernel, e.g.
the fast Fourier transform 
\cite{long2006fast, abert2012fast, yuan1992fast, garcía2003accurate}, the fast multipole method \cite{blue1991using},
nonuniform grid method \cite{livshitz2009nonuniform}, and the tensor grid method \cite{exl2012fast}.
The best numerical methods reach complexity between $N$ and $N \log N$, where $N$ is the number of
unknowns.

The paper is organized as follows. 
In Section~2, we describe the PDE formulation of the LL equation.
In Section~3, we present the MFD method.
In Subection~\ref{sec:stray}, the computation of the stray field is reviewed based on \cite{wang2006simulations}. 
In Section~\ref{sec:stability}, the stability of the explicit and implicit schemes is analyzed. 
Finally, in Section~\ref{sec:numerical}, we verify and validate the proposed schemes using analytical 
solutions, the NIST $\mu$mag standard problem $4$, and the test showing formation of domain wall 
structures in a thin film.

%%%%%%%%%%%%%%%%%%%%%%%%%%%%%%%%%%%
\section{Problem formulation}
\setcounter{equation}{0}

The dynamics of the magnetic distribution in a ferromagnetic material occupying a region 
$\Omega \subset \mathbb{R}^d$ where $d=2$, or $3$, is governed by the LL equation,
see, e.g. \eqref{eq:LLnist}.
After its normalization, we obtain the following PDE for the magnetization 
$\mb{m} : \Omega \times [0,T]\to \mathbb{R}^3$:
\begin{equation}\label{eq:LL}
  \der{\mb{m}}{t} = - \mb{m} \times \mb{h} - \alpha \mb{m} \times ( \mb{m} \times \mb{h}),
\end{equation}
where $\alpha$ is the dimensionless damping parameter and $\mb{h}$ is the effective field.
The first term on the right hand side is the gyromagnetic term and the second term is the damping term.
The problem is closed by imposing the Neumann or Dirichlet boundary conditions and initial
conditions.

It is immediately to see that $|\mb{m}| $ is constant in time, so we assume that $|\mb{m}|=1$.
The effective field is defined as the functional derivative of the LL energy density:
\begin{equation}\label{eq:LLenergy}
  \mb{h}(\mb{m}) = -\derr{E(\mb{m})}{\mb{m}}
   = \eta \triangle \mb{m} - Q(m_2 e_2 + m_3 e_3) +\mb{h_s}(\mb{m})+\mb{h_e},
\end{equation}
where $\eta$ is the exchange constant, $Q$ is an anisotropy constant, $\mb{h}_s$ is the stray field, 
and $\mb{h}_e$ is an external field.
%and the energy density is given by $E(\mb{m}) = \int _\Omega \mc{E}(\mb{m})\dx$.
Let us collect the low-order terms (with respect to $\Delta \mb{m}$) in a single variable
\begin{equation}\label{eq:low}
  \mb{\bar h}(\mb{m}) = - Q(m_2 e_2 + m_3 e_3) +\mb{h_s}(\mb{m})+\mb{h_e}.
\end{equation}
These terms are usually regarded as the low-order terms when considering mathematical properties 
such as the existence and regularity of the solution \cite{alouges2006convergence}.
Since constant $\eta$ is not critical for the description of the mimetic scheme, we set $\eta=1$.

The stray field is given by $\mb{h_s}=- \nabla \phi$. 
Let $\Omega^c$ denote the complement of $\Omega$.
Then, the potential $\phi$ satisfies (see \cite{garcia2007numerical} for more detail):
\begin{equation}\label{eq:stray}
\begin{aligned}
  \triangle \phi &= 
  \begin{cases}
    \nabla \cdot \mb{m} &\text{ in } \Omega, \\
    0 &\text{ on }  \Omega^c, \\
  \end{cases}\\[0.5ex]
  [\phi]_{\partial \Omega} &= 0, \\[0.5ex]
  [\nabla \phi \cdot \bm{n}]_{\partial \Omega} &= - \mb{m} \cdot \mb{n},
\end{aligned}
\end{equation}
where $[v]_{\partial \Omega}$ denotes a jump of function $v$ across the domain boundary,
and $\bm{n}$ is the unit normal vector.
Hence, the stray field $\mb{h_s}$ is given by
\begin{equation}\label{eq:hs}
  \mb{h_s} (x) 
  =- \frac{1}{4\pi} \nabla \int_\Omega \nabla\left(\frac{1}{|x-y|}\right) \cdot \mb{m}(y) \;{\rm d}y.
  %= \frac{1}{4\pi} \nabla \{\int_\Omega \frac{\nabla \cdot m(y) }{|x-y|} dy 
  % -\int_{\partial \Omega} \frac{m(y) \cdot \vec n}{|x-y|} dS(y) \}
\end{equation}

Note, that there are several equivalent forms of the LL equation, e.g. the Landau-Lifshitz-Gilbert equation:
\begin{equation}\label{eq:LLG}
  \der{\mb{m}}{t} - \alpha \mb{m} \times \der{\mb{m}}{t} 
  = -(1+  \alpha^2)  ( \mb{m} \times \mb{h}).   
\end{equation}
Another equivalent form, used in \cite{alouges2006convergence} to develop a numerical scheme, is
\begin{equation}\label{eq:LLG2}
  \alpha \der{\mb{m}}{t} + \mb{m} \times \der{\mb{m}}{t}
  = (1+  \alpha^2)  (\mb{h} - (\mb{h} \cdot \mb{m} ) \mb{m}).
\end{equation}
In a special case of $\mb{\bar h} = 0$, we have $\mb{h} = \triangle \mb{m}$.
Then, the simplified LL equation (\ref{eq:LL}) has more equivalent forms:
\begin{equation}\label{eq:LLsimplified}
\begin{aligned}
  \der{\mb{m}}{t} =& - \mb{m} \times \Delta\mb{ m} - \alpha \mb{m} \times ( \mb{m} \times\Delta \mb{ m}) \\[0.0ex]
  = & -\mb{m} \times \Delta\mb{ m} + \alpha \Delta \mb{ m} - \alpha( \mb{m} \cdot \Delta \mb{ m})   \mb{m} \\[1ex]
  = & -\mb{m} \times \Delta\mb{ m} + \alpha \Delta \mb{ m} + \alpha|\nabla  \mb{m}|^2 \mb{m}.
\end{aligned}
\end{equation}
Here, we used the vector identity 
$\mb{a}\times(\mb{b}\times\mb{c}) = (\mb{a}\cdot \mb{c})\; \mb{b} - (\mb{a} \cdot \mb{b})\;\mb{c}$,
$|\mb{m}|=1$, and 
$$
  \mb{m} \cdot \Frac{\partial\mb{m}}{\partial u} = 0,
  \quad u \in \{x,\,y,\,z\}.
$$
If we consider only the damping term on the right-hand side of (\ref{eq:LLsimplified}) we obtain $\der{\mb{m}}{t} =- \alpha \mb{m} \times ( \mb{m} \times \triangle \mb{m})$,
which is called the harmonic map heat flow into $\mathbb{S}^2$ \cite{gustafson2010asymptotic}. 
If we consider only the gyromagnetic term on the right-hand side of (\ref{eq:LLsimplified}), we obtain 
$\der{\mb{m}}{t} = - \mb{m} \times \triangle \mb{m}$, which is called the Schr\"{o}dinger map, 
a geometric generalization of the linear Schr\"{o}dinger equation \cite{gustafson2010asymptotic}. 
By using the LL equation to design a numerical scheme, instead of the Landau-Lifshitz-Gilbert 
equation (\ref{eq:LLG}) or (\ref{eq:LLG2}), we can immediately apply it to the harmonic map heat flow 
and the Schr\"odinger map.

%%%%%%%%%%%%%%%%%%%%%%%%%%%%%%%%%%%%%%%%%%%%%%%%%%%%%%%%%%%%%%%%%%%%%
\section{Mimetic discretization of the Landau-Lifshitz equation}
\label{sec:mimetic}
\setcounter{equation}{0}

In this section, we apply the mimetic finite difference (MFD) method to the LL equation.
Let us introduce the magnetic flux tensor $\mb{p} = -\mb{\nabla} \mb{m}$.
Then,
$$
  \frac{\partial \mb{m}}{\partial t} 
  = \mb{m} \times \div \mb{p} + \alpha \mb{m} \times(\mb{m} \times  {\rm\bf div \;} \mb{p})
  + \mb{f}(\mb{m}).
$$
where
\begin{equation}\label{eq:extra}
  \mb{f}(\mb{m}) 
  =- \mb{m} \times \mb{\bar h}( \mb{m})- \alpha \mb{m} \times(\mb{m} \times \mb{\bar h}(\mb{m}))
\end{equation}
corresponds to the low-order terms (\ref{eq:low}).
The MFD method solves simultaneously for both $\mb{m}$ and $\mb{p}$.
It mimics duality of the divergence and gradient operators in the discrete setting.
This property is used in the stability analysis.
The MFD method works on arbitrary polygonal or polyhedral mesh which provides
enormous flexibility for modeling non-rectangular mechanical devices.

Let the computational domain $\Omega$ be decomposed into $N_E$ non-overlapping polygonal 
or polyhedral elements $E$ with the maximum diameter $h$. 
Let $N_F$ denote the total number of mesh edges (faces in 3D).
We use $|E|$ to denote the area (volume in 3D) of $E$.
Similarly, $|f|$ denotes the length of mesh edge $f$ (area of mesh face $f$ in 3D).
Let $\bm{n}_E$ be the unit vector normal to $\partial E$.

To define degrees of freedom for the mimetic scheme, we assume that $\mb{m} \in {\cal Q}$ 
and $\mb{p} \in {\cal F}$, where ${\cal Q} = (L^2(\Omega))^3$ and
\begin{equation}
  {\cal F} = \{\mb{p} \;|\; \mb{p} \in (L^s(\Omega))^{d\times 3},\ s>2,\ \div \mb{p} \in (L^2(\Omega))^3\}.
\end{equation}

%%%%%%%%%%%%%%%%%%%%%%%%%%%%%%%%%%%%%%%%%%%%%%%%%%%%%%%%%%%%%%%%%%%%%
\subsection{Global mimetic formulation}

Let $\mb{m} = (m_x,\,m_y,\,m_z)$.
The degrees of freedom for each component of the magnetization 
are associated with elements $E$ and denoted as $m_{x,E}$, $m_{y,E}$, and $m_{z,E}$.
They represent the mean values of $\mb{m}$:
$$
  m_{u,E} = \frac{1}{|E|} \int_E m_u \dx,
  \quad u \in \{x,\,y,\,z\}.
$$
Thus, control of the magnetization magnitude reduces to simple cell-based constraints $|\mb{m}_E| = 1$.
Consider the vector space 
\begin{equation} \label{eq:Qh}
  \mathcal{Q}^h = \left\{ m^h_u\colon\  m^h_u = (m_{u, E_1}, \cdots, m_{u,E_{N_E}} )^T \right\}.
\end{equation}
The dimension of this space is equal to the number of mesh elements.
Then, the discrete magnetization $\mb{m}^h = (m_x^h,\,m_y^h,\,m_z^h)$ with $m_u^h \in {\cal Q}^h$.

Let $\mb{p} = (\mb{p}_x,\,\mb{p}_y,\,\mb{p}_z)$.
The degrees of freedom for each component of the magnetic flux are associated with
mesh edges $f$ (faces in 3D) and denoted as $p_{x,E,f}$, $p_{y,E,f}$ and $p_{z,E,f}$.
They represent the mean normal flux across edge $f$ of element $E$:
$$
  p_{u,E,f} = \frac{1}{|f|} \int_f \mb{p}_u \cdot \bm{n}_E \dx,
  \quad u \in \{x,\,y,\,z\}.
$$
We need a few notations for local groups of degrees of freedom.
Let $p_{u,E}$ be the vector of degrees of freedom associated with element $E$ and
$\mb{p}_E = (p_{x,E},\,p_{y,E},\,p_{z,E})^T$.
Consider the vector space 
\begin{equation} \label{eq:Fh}
  \mathcal{F}^h = \left\{ p^h_u\colon\  p^h_u = (p_{u, E_1}, \cdots, p_{u,E_{N_E}} )^T \right\}.
\end{equation}
The dimension of this space is equal to twice the number of internal mesh edges (faces in 3D)
plus the number of boundary edges.
Then, the discrete magnetic flux $\mb{p}^h = (p_x^h,\,p_y^h,\,p_z^h)$ with $p_u^h \in {\cal F}^h$.

In the global mimetic formulation we consider the reduced space for discrete fluxes still 
denoted by $\mathcal{F}^h$.
Let $f$ be an internal edge shared by two elements $E_1$ and $E_2$. 
The reduced space satisfies the flux continuity constraints
\begin{equation}\label{eq:continuity}
  p_{u,E_1,f} + p_{u,E_2,f} = 0
\end{equation}
for all internal edges. The reduced space allows us to work with exterior normal
vectors $\mb{n}_E$ which simplifies some formulas.
In a computer code, the flux continuity constraints are used to reduce the problem size.

The degrees of freedom were chosen to define  the discrete divergence operator easily.
For each component $u$ of the magnetic flux, the divergence theorem for element $E$ leads to
the discrete divergence operator:
\begin{equation} \label{eq:div}
  \DIV_E \;p_{u,E} = \frac{1}{|E|}\sum_{f \in \partial E} |f|\, p_{u,E,f}.
\end{equation} 
Let
$$ 
  \bDIV_E \;\mb{p}_{E} = (\DIV_E \;p_{x,E},\, \DIV_E \;p_{y,E},\, \DIV_E \;p_{z,E})^T.
$$

The MFD method builds the discrete gradient operator from the discrete duality principle.
Recall that in the continuum setting, under the homogeneous boundary conditions, we have the Green formula:
$$
  \int_\Omega \mb{m} \cdot \div \mb{p} \dx
  = -\int_\Omega \nabla \mb{m} \colon \mb{p} \dx.
$$
It states that the gradient operator is negatively adjoint to the divergence operator
with respect to the $L^2$-inner products.
In the MFD method, we mimic this formula.
First, we replace the $L^2$-inner products by discrete inner products in spaces of the
degrees of freedom.
In the space of discrete magnetizations, we define the following inner product:
\begin{equation}\label{innerQ}
  \big[\mb{m}^h,\,\mb{w}^h\big]_{{\cal Q}} 
  = \sum_{u \in \{x,y,z\}} \big[{m^h_u},\, {w^h_u}\big]_{{\cal Q}},\quad
  \big[m^h_u,\, w^h_u\big]_{{\cal Q}} 
  = \sum_{E \in \Omega_h} \big[m^h_{u,E},\, w^h_{u,E}\big]_{{\cal Q},E}.
\end{equation}
Since we have only one degree of freedom per mesh element, we have
$\big[m^h_{u,E},\, w^h_{u,E}\big]_{{\cal Q},E} = |E|\,m_{u,E}\, w_{u,E}$.
In the space of discrete fluxes, we define the following inner product:
\begin{equation}\label{innerF}
  \big[\mb{p}^h,\,\mb{q}^h\big]_{{\cal F}} 
  = \sum_{u \in \{x,y,z\}} \big[{p^h_u},\, {q^h_u}\big]_{{\cal F}},\quad
  \big[p^h_u,\, q^h_u\big]_{{\cal F}} 
  = \sum_{E \in \Omega_h} \big[p_{u,E},\, q_{u,E}\big]_{{\cal F},E},
\end{equation}
where $\big[\cdot,\, \cdot\big]_{{\cal F},E}$ is an element-based inner product that 
requires special construction discussed later.
The mimetic gradient operator is defined implicitly from the discrete duality property:
$$
  \big[\mb{m}^h,\,\bDIV\, \mb{p}^h\big]_{{\cal Q}}
  = -\big[\bGRAD\,\mb{m}^h,\,\mb{p}^h\big]_{{\cal F}} \qquad \forall \mb{m}^h, \mb{p}^h.
$$
Due to basic properties of the inner products, the mimetic gradient is defined uniquely.

The semi-discrete mimetic formulation is to find $\mb{m}^h$ and $\mb{p}^h$ such that
$$
\begin{array}{rcl}
  \mb{p}^h &\!\!=\!\!& -\bGRAD\, \mb{m}^h,\\[1ex]
  \Frac{\partial \mb{m}^h}{\partial t} + \alpha\, \bDIV\, \mb{p}^h 
  &\!\!=\!\!& \mb{m}^h \times \bDIV \mb{p}^h + \alpha\, (\mb{m}^h \cdot \bDIV\, \mb{p}^h)\,\mb{m}^h
  + \mb{f}^h(\mb{m}^h)
\end{array}
$$
where $\mb{f}^h(\mb{m}^h)$ is a discretization of equation (\ref{eq:extra}).
Non-homogeneous boundary conditions can be incorporated in the MFD framework using
the approach described in \cite{Hyman-Shashkov:1998}.
In the next section, we describe another way to include boundary conditions.

%%%%%%%%%%%%%%%%%%%%%%%%%%%%%%%%%%%%%%%%%%%%%%%%%%%%%%%%%%%%%%%%%%%%%
\subsection{Local mimetic formulation}

For the local mimetic formulation, the starting point is the local
Green formula:
\begin{equation}
  \int_E \mb{m} \cdot \div \mb{p} \dx 
   = -\int_E \nabla \mb{m} \colon \mb{p} \dx 
     +\int_{\partial E} (\mb{p} \cdot \mb{n}) \cdot \mb{m} \dx
\end{equation}
To discretize the last integral, we use additional degrees of freedom for magnetization 
on mesh edges $f$ that we denote by $\mb{m}_f = (m_{x,f},\, m_{y,f},\, m_{z,f})$. 
Let $\widetilde{\mb{m}}_E$ be the vector of the additional degrees of freedom associated with 
element $E$.
The total number of these degrees of freedom is equal to the number of mesh edges, $N_F$,
times three.
The local mimetic gradient operator is defined implicitly from the discrete duality property:
\begin{equation}\label{localGreen-h}
  \big[\mb{m}_E,\,\bDIV_E\, \mb{p}_E\big]_{{\cal Q},E}
  = -\big[\bGRAD_E\begin{pmatrix}\mb{m}_E\\ \widetilde{\mb{m}}_E \end{pmatrix},\,\mb{p}_E\big]_{{\cal F},E}
    +\sum\limits_{f \in \partial E} |f|\, \mb{p}_f\cdot \mb{m}_f 
\end{equation}
for all $\mb{m}_E$, $\widetilde{\mb{m}}_E$, and $\mb{p}_E$. 
Due to basic properties of the inner products, the local mimetic gradient operator is defined uniquely.
Moreover, by our assumption, all inner products are sums of inner products for
vector components.
This allows us to define a component-wise gradient operator via
$$
  \big[m_{u,E},\,\DIV_E\, p_{u,E}\big]_{{\cal Q},E}
  = -\big[\GRAD_E\begin{pmatrix}m_{u,E}\\ \widetilde{m}_{u,E} \end{pmatrix},\,p_{u,E}\big]_{{\cal F},E}
    +\sum\limits_{f \in \partial E} |f|\, p_{u,E,f}\cdot m_{u,f}.
$$
The last formula can be simplified if we define inner product matrices 
$\mathbb{M}_{{\cal Q},E}$ and $\mathbb{M}_{{\cal F},E}$:
$$
  \big[m_{u,E},\,w_{u,E}\big]_{{\cal Q},E}
  = m_{u,E}\, \mathbb{M}_{{\cal Q},E}\, w_{u,E},
  \qquad
  \big[p_{u,E},\, q_{u,E}\big]_{{\cal F},E}
  = p_{u,E}\, \mathbb{M}_{{\cal F},E}\, q_{u,E}.
$$
The first matrix is simply the number $|E|$.
The construction of the second matrix is more involved and is discussed in subsection \ref{subsec:W}.
If the boundary of element $E$ has $n$ faces $f_i$, we obtain the following explicit formula
for the mimetic gradient:
\begin{equation} \label{eq:pW}
  \GRAD_E
  \begin{pmatrix}m_{u,E}\\ \widetilde{m}_{u,E} \end{pmatrix}
  = 
  \mathbb{M}_{{\cal F},E}^{-1}
  \begin{pmatrix}
   |f_1|\, (m_{u,f_1}-m_{u,E})\\ 
   |f_2|\, (m_{u,f_2}-m_{u,E})\\ 
   \vdots \\ 
   |f_n|\, (m_{u,f_n}-m_{u,E})
  \end{pmatrix}.
\end{equation}

The semi-discrete mimetic formulation is to find $\mb{m}_E$, $\widetilde{\mb{m}}_E$, and 
$\mb{p}_E$, for $E \in \Omega^h$, such that
\begin{equation}\label{eq:mixed}
\begin{array}{rcl}
  \mb{p}_E &\!\!=\!\!& -\bGRAD_E\begin{pmatrix}\mb{m}_E\\ \widetilde{\mb{m}}_E \end{pmatrix}, \\[1ex]
  \Frac{\partial \mb{m}_E}{\partial t} + \alpha\, \bDIV_E\, \mb{p}_E
  &\!\!=\!\!& \mb{m}_E \times \bDIV_E \mb{p}_E + \alpha\, (\mb{m}_E \cdot \bDIV_E\, \mb{p}_E)\,\mb{m}_E
  + \mb{f}_E(\mb{m}_E),
\end{array}
\end{equation}
subject to continuity \eqref{eq:continuity}, boundary, and initial conditions.
The Dirichlet boundary conditions are imposed by prescribing given values to the 
auxiliary magnetization unknowns on mesh edges.
The Neumann boundary conditions are imposed by setting magnetic fluxes to given
values.

The local mimetic formulation implies the global one.
Equivalence of the local and global gradient operators can be shown by summing 
up equations \eqref{localGreen-h} and observing that interface term are
canceled out due to flux continuity conditions \eqref{eq:continuity}.
The local formulation is more convenient for a computer implementation.
The global formulation is more convenient for convergence analysis.

%%%%%%%%%%%%%%%%%%%%%%%%%%%%%%%%%%%%%%%%%%%%%%%%%%%%%%%%%%%%%%%%%%%%%%
\subsection{Implicit-explicit time discretization}
To discretize time derivative in \eqref{eq:mixed}, we consider the $\theta$-scheme.
Let $k$ denote the time step and superscript $j$ denote the time moment $t^j = j\,k$.
Then, the $\theta$-scheme is
\begin{equation}\label{eq:dmdt}
\begin{aligned} 
  \mb{p}_E^{j+\theta} 
  &= -\bGRAD_E\begin{pmatrix}\mb{m}_E^{j+\theta}\\ \widetilde{\mb{m}}_E^{j+\theta} \end{pmatrix}, \\
  \frac{\mb{m}_E^{j+1} - \mb{m}_E^j}{k} 
  + \alpha \bDIV_E\;p^{j+\theta}_E
  &= \mb{m}_{E}^j \times \bDIV_{E}\;\mb{p}_E^{j+\theta}
   +\alpha\, (\mb{m}_E^j \cdot (\bDIV_E\; \mb{p}^{j+\theta}_E)) \mb{m}_E^j
   +\mb{f}^j_E(\mb{m}^j_E).
\end{aligned}
\end{equation}
Note that $\theta =0$ for the explicit scheme and $\theta =1$ for the implicit scheme.
The later scheme is not fully implicit and a single linear solver is required to
advance the solution.
Hereafter, we consider only $\theta = 1$.
Hereafter, we drop superscript `$h$' and write $\mb{m}^j$ instead of $\mb{m}^{h,j}$.
Inserting the first equation in the second one and in the flux continuity conditions (multiplied
by $|f|$ for better symmetry),
and imposing boundary and initial conditions, we obtain a system of algebraic equations
for cell-based and edge-based magnetizations:
\begin{equation}\label{eq:linearsystem}
  \mbb A^j 
  \begin{pmatrix} \mb{m}^{j+1}\\[1ex] \widetilde{\mb{m}}^{j+1} \end{pmatrix}
  =
  \begin{pmatrix} \mb{b}_1^j \\[1ex] \mb{b}_2^j \end{pmatrix}.
\end{equation}
The right-hand side vector depends on $\mb{m}^j$, external field, stray field,
and boundary conditions.
The stiffness matrix $\mathbb A^j$ is sparse and could be written as the sum of local matrices,
\begin{equation}
  \mbb A^j = \sum_{E \in \Omega^h} \mbb N_E\, \mbb A^j_E\, \mbb N_E^T.
\end{equation}
where $\mbb N_E$ is the conventional assembly matrix which maps local indices to global indices.
We present detailed structure of the local matrix which is important for analyzing its
structural properties and selecting optimal solver.
Let us consider an element $E$ with $n$ edges located strictly inside the computational domain.
The corresponding local matrix has size $3 (n+1) \times 3(n+1)$ and the following block structure:
\begin{equation}
  \mbb{A}_E^j =
  \begin{pmatrix}
    \mbb{A}^{EE} & \mbb{A}^{Ef}\\[0.5ex]
    \mbb{A}^{fE} & \mbb{A}^{ff}
  \end{pmatrix}
\end{equation}
where the first block row corresponds to three cell-based magnetizations.
Let $\mbb{I}$ denote a generic identity matrix, $\mbb{C}_E$ be the $n\times n$ diagonal matrix,
$\mbb{C}_E = {\rm diag}\{|f_1|,\ldots,|f_n|\}$, and $\mb{e}=(1,1,\ldots,1)^T$.
Using the tensor-product notation to represent matrices, we have
\begin{equation}
\begin{array}{rcll}
  \mbb{A}^{EE} &\!\!=\!\!& \mbb{I}
    +\Frac{k}{|E|} (\mb{e}^T \mbb{C}_E\, \mbb{M}_{{\cal F},E}^{-1}\,\mbb{C}_E\, \mb{e})\, \mbb{\hat  A}^j,
  \qquad
  &\mbb{A}^{Ef} = -\Frac{k}{|E|} \big(\mb{e}^T \mbb{C}_E\, \mbb{M}_{{\cal F},E}^{-1}\, \mbb{C}_E\big)\otimes\mbb{\hat A}^j\\[2ex]
  \mbb{A}^{fE} &\!\!=\!\!& -\big(\mbb{C}_E\,\mbb{M}_{{\cal F},E}^{-1}\, \mbb{C}_E\,\mb{e}\big) \otimes \mbb{I},
  \qquad
  &\mbb{A}^{ff} = \big(\mbb{C}_E\,\mbb{M}_{{\cal F},E}^{-1}\, \mbb{C}_E\big) \otimes \mbb{I},
\end{array}
\end{equation}
and
\begin{equation}
  \mbb{\hat A}^j = \alpha\,\mbb{I} - \alpha\, \mb{m}^j_E \, (\mb{m}^j_E)^T 
  - \begin{pmatrix}
    0         & - m_{z,E}^j  & m_{y,E}^j \\[0.5ex]
    m_{z,E}^j &  0          &- m_{x,E}^j \\[0.5ex]
   - m_{y,E}^j & m_{x,E}^j  &  0
\end{pmatrix}.
\end{equation}
 
To preserve the geometric constraint $|\mb{m}_E^{j+1}| = 1$, we need to modify 
the solution by projecting it onto the unit sphere.
This leads to the following algorithm.

\medskip
\medskip
\noindent{\bf Algorithm 1.} For a given final time $T > 0$, set $J = [{T \over k}]$.
\begin{enumerate}
\item Set an initial discrete magnetization $\mb{m}^{0}$ at the centers of mesh elements.
\item For $j=0, \dots, J-1$,
  \begin{enumerate}
  \item Form and solve the system \eqref{eq:linearsystem}. 
    Let $\begin{pmatrix} \hat{\mb{m}}^{j+1}\\ \widetilde{\mb{m}}^{j+1}\end{pmatrix}$ denote the solution.
  \item Renormalize the element-centered magnetizations:
    $$
      \mb{m}^{j+1}_E \, :=\, \Frac{\hat{\mb{m}}^{j+1}_E}{|\hat{\mb{m}}^{j+1}_E|}, \quad 
      \quad \forall  E \in \Omega^h.
    $$
  \end{enumerate}
\end{enumerate}
 
Finally, note that the scheme require only matrix $\mbb{M}_{{\cal F},E}^{-1}$.
In the next section, we show that it could be calculated directly for the same cost as 
matrix $\mbb{M}_{{\cal F},E}$.

%%%%%%%%%%%%%%%%%%%%%%%%%%%%%%%%%%%%%%%%%%%%%%%%%%%%%%%%%%%%%%%%%%%%%%
\subsection{Construction of matrices $\mbb{M}_{{\cal F},E}$ and  $\mbb{M}_{{\cal F},E}^{-1}$}
\label{subsec:W}
The construction of the matrices $\mbb{M}_{{\cal F},E}$ and  $\mbb{M}_{{\cal F},E}^{-1}$ is 
based on the consistency and stability conditions discussed in detail in \cite{brezzi2005family}.
Here, we briefly summarize the underlying ideas.
The inner product matrix $\mbb{M}_{{\cal F},E}$ must represent an accurate quadrature
for the continuum $L^2$ inner product of two functions.
More specifically, we require that the following consistency condition (also known as the 
exactness property) holds:
$$
  \big(p_{u,E}^0\big)^T\,\mbb{M}_{{\cal F},E}\,q_{u,E}
  = [p_{u,E}^0,\,q_{u,E}]_{{\cal F},E}
  = \int_E \mb{p}^0_u \cdot \mb{q}_u \dx
$$
for any constant vector function $\mb{p}^0_u$ and any sufficiently smooth function $\mb{q}_u$ 
with constant divergence and constant normal components on element edges. 
Here $p_{u,E}^0$ and $q_{u,E}$ are vectors of the degrees of freedom for functions 
$\mb{p}^0_u$ and $\mb{q}_u$, respectively.
Writing $\mb{p}^0_u$ as the gradient of a linear function $m_u^1$ with mean zero value on $E$
and integrating by parts, we obtain:
$$
  \int_E \mb{p}^0_u \cdot \mb{q}_u \dx 
  = -\int_E  m_u^1\, {\rm div}\, \mb{q}_u \dx 
    +\int_{\partial E} (\mb{q}_u \cdot \mb{n}_E)\, m_u^1 \dx
  = \sum\limits_{f \in \partial E} q_{u,E,f} \int_f m_u^1 \dx.
$$
For a given $\mb{p}^0_u$, this formula uses the degrees of freedom and 
the computable edge integrals.
Taking three linearly independent constant vector functions (such as $\mb{p}^0_u = (1,0,0)^T$),
inserting them in the exactness property, and using the definition of the degrees of freedom,
we obtain the matrix equation
\begin{equation}\label{eq:mnr}
  \mathbb{M}_{{\cal F},E}\,\mathbb{N} = \mathbb{R},
\end{equation}
where
$$
\mathbb{N}= \begin{pmatrix}
  n_{x_{f_1}} & n_{y_{f_1}} & n_{z_{f_1}}\\
  n_{x_{f_2}} & n_{y_{f_2}} & n_{z_{f_2}}\\
  \vdots  & \vdots & \vdots \\
  n_{x_{f_m}} & n_{y_{f_m}} & n_{z_{f_m}}
\end{pmatrix},
\quad
\mathbb{R}= \begin{pmatrix}
  |f_1|(x_{f_1}-x_E) & |f_1|(y_{f_1}-y_E) & |f_1|(z_{f_1}-z_E) \\
  |f_2|(x_{f_2}-x_E) & |f_2|(y_{f_2}-y_E) & |f_2|(z_{f_2}-z_E) \\
  \vdots  & \vdots & \vdots \\
  |f_m|(x_{f_m}-x_E) & |f_m|(y_{f_m}-y_E) & |f_m|(z_{f_m}-z_E)
\end{pmatrix}.
$$
Here $n_{u_f}$ is the $u$-th component of the outward unit normal vector to edge $f$ of $E$,
$(x_f,\,y_f,\,z_f)$ is the edge centroid, and $(x_E,\,y_E,\,z_E)$ is the element centroid.
It was shown in \cite{brezzi2005family} that the symmetric positive definite matrix
\begin{equation}\label{eq:MFE}
  \mathbb M_{{\cal F},E} 
  = \frac{1}{|E|} \mathbb{R} \mathbb{R}^T
  + \gamma\, (\mbb{I} - \mathbb{N}\, (\mathbb{N}^T\, \mathbb{N})^{-1}\, \mathbb{N}^T),
  \qquad \gamma > 0,
\end{equation}
is a solution to  matrix equation \eqref{eq:mnr}.
The matrix $\mathbb M_{{\cal F},E}$ looks like a mass matrix when both terms in \eqref{eq:MFE} 
are scaled equally with respect to local mesh size.
Since the second term includes the orthogonal projection, we can simply set 
$\gamma = |E|$ or $\gamma = \frac{1}{2\,|E|} \text{ trace} (\mathbb{R}\, \mathbb{R}^T) $.

The general solution to \eqref{eq:mnr} is
\begin{equation}\label{eq:full_family}
  \mathbb M_{{\cal F},E} 
  = \frac{1}{|E|} \mathbb{R} \mathbb{R}^T
  + (\mbb{I} - \mathbb{N}\, (\mathbb{N}^T\, \mathbb{N})^{-1}\, \mathbb{N}^T)
  \,\mathbb{G}\, (\mbb{I} - \mathbb{N}\, (\mathbb{N}^T\, \mathbb{N})^{-1}\, \mathbb{N}^T)
\end{equation}
where $\mathbb{G}$ is a symmetric positive definite matrix.
Mathematically, the selection of this matrix must satisfy the {\it stability condition}:
There exists two positive constants $c_0, C_0>0$, independent of $h$ and $E$, such that, 
for every $q_u \in \mathcal{F}^h_E$ and every $E \in \Omega^h$, we have
\begin{equation}\label{eq:stability}
  c_0\, |E|\,q_{u,E}^T\, q_{u,E} 
  \le q_{u,E}^T\, \mathbb{M}_{{\cal F},E}\,q_{u,E}
  = [q_{u,E},\,q_{u,E}]_{{\cal F},E}
  \le C_0\, |E|\, q_{u,E}^T\, q_{u,E}.
\end{equation}

Note that the matrix equation \eqref{eq:mnr} can be also written as 
\begin{equation}\label{eq:wrn}
  \mathbb{M}_{{\cal F},E}^{-1}\,\mathbb{R} = \mathbb{N}.
\end{equation}
Since the role of matrices $\mathbb{R}$ and $\mathbb{N}$ has changed but the equation structure did not,
we can write immediately the general solution to this equation:
\begin{equation}\label{eq:WFE}
  \mathbb{M}_{{\cal F},E}^{-1} 
  = \frac{1}{|E|} \mathbb{N} \mathbb{N}^T
  + (\mbb{I} - \mathbb{R}\, (\mathbb{R}^T\, \mathbb{R})^{-1}\, \mathbb{R}^T)\,
    \tilde{\mathbb{G}}\, (\mbb{I} - \mathbb{R}\, (\mathbb{R}^T\, \mathbb{R})^{-1}\, \mathbb{R}^T),
\end{equation}
where $\tilde{\mathbb{G}}$ is symmetric positive definite matrix.
In a computer code, we can use a scalar matrix 
$\tilde{\mathbb{G}} = \tilde \gamma \mathbb{I}$ where 
$\tilde \gamma = \frac{1}{|E|}$ or $\tilde \gamma = \frac{1}{2\,|E|} \text{ trace} (\mathbb{N}\, \mathbb{N}^T) $.

\begin{remark}
Since both matrix equations \eqref{eq:mnr} and \eqref{eq:wrn} have multiple solutions, 
formulas \eqref{eq:MFE} and \eqref{eq:WFE} with $\gamma = \tilde \gamma^{-1} = |E|$ represent 
non-related members from two families of solutions.
\end{remark}

%%%%%%%%%%%%%%%%%%%%%%%%%%%%%%%%%%%%%%%%%%%%%%%%%%%%%%%%%%%%%%%%%%%%%
\subsection{Computation of the stray field}
\label{sec:stray}

In this subsection, we discuss how to compute the stray field $\mb{h}_s(\mb{m})$
on various meshes and specifically on uniform meshes over rectangular domains.
For general geometries, we refer to the stray field calculation survey \cite{abert2013numerical}.

We can break the integral in equation (\ref{eq:hs}) into sum over mesh elements,
$$ %\begin{equation} \label{eq:hssum}
\begin{aligned}
  \mb{h}_s(x) &=-\frac{1}{4\pi}\!\sum_{E\in\Omega^h} \nabla \int_E \nabla(\frac{ 1}{|x-y|})  \cdot \mb{m}(y) \dy \\
  &= \frac{1}{4\pi}\!\sum_{E\in\Omega^h} \nabla \left\{\int_E \frac{\nabla \cdot \mb{m}(y)}{|x-y|} \dy 
   -\int_{\partial E} \frac{\mb{m}(y) \cdot \mb{n}_{E}}{|x-y|} \dy \right\}.
\end{aligned}
$$ %\end{equation}
The discrete magnetization is constant in each element $E$; hence, its divergence is zero 
and we obtain the following formula for calculating the discrete stray field:
\begin{equation}
  \mb{h}_{s,E} = \frac{1}{4\pi} \sum_{E\in\Omega^h} 
    \int_{\partial E} \frac{x_E-y}{|x_E-y|^3}\; \mb{m}_E(y) \cdot \mb{n}_E \dy.
\end{equation}

Efficient implementation of the above summation can be done on the uniform mesh over 
a rectangular domain.
We applied this method for a few micromagnetic simulations considered
in subsection \ref{subsec:nist} and \ref{subsec:domainwall}.
Let $E_{ijk}$ be a cuboid element,
$$
  E_{ijk} = [x_{i-\frac{1}{2}},x_{i+\frac{1}{2}}] \times 
            [y_{j-\frac{1}{2}},y_{j+\frac{1}{2}}] \times 
            [z_{k-\frac{1}{2}},z_{k+\frac{1}{2}}],
$$
where $x_i = (i-0.5)h_x$, $y_j = (j-0.5)h_y$, and $z_k = (k-0.5)h_z$.
Then, the above summation can be rewritten using a $3\times3$ demagnetization tensor $\mbb{K}$:
\begin{equation}
  \mb{h}_{s,E_{ijk}} = \sum_{p,q,r} \mbb{K}_{i-p,j-q, k-r} \mb{m}_{E_{pqr}}.
\end{equation}
The demagnetization tensor can be computed analytically:
$$
\begin{aligned}
  \mbb{K}_{i-p,j-q, k-r} = -\frac{1}{4 \pi} 
  \sum_{s_x,s_y,s_z} s_x s_y s_z 
  \begin{pmatrix}
    \arctan (\Frac{R_y R_z}{R_x R}) & - \log(R+ R_z) & - \log(R+ R_y)  \\[0.5ex]
    -\log(R+ R_z)  & \arctan (\frac{R_x R_z}{R_y R}) & - \log(R+ R_x)  \\[0.5ex]
    -\log(R+ R_y)  & - \log(R+ R_x)  & \arctan (\Frac{R_x R_y}{R_z R}) 
  \end{pmatrix}
\end{aligned}
$$
where $s_x, s_y, s_z= \pm 1$, $R_x = \frac{h_x}{2}-s_x(i-p)h_x$, 
$R_y = \frac{h_y}{2}-s_y(j-q)h_y$, $R_z = \frac{h_z}{2}  -s_z(k-r)h_z$ and 
$R = \sqrt{R_x^2+R_y^2+R_z^2}$.
In particular, if there is no discretization in the $z$ direction, such as in a thin film, 
then we have $\mbb{K}_{xz} = \mbb{K}_{zx} = \mbb{K}_{yz} = \mbb{K}_{zy}=0$. 
The stray field can be calculated using the FFT, see \cite{hayashi1996calculation,wang2006simulations} for more details.

%%%%%%%%%%%%%%%%%%%%%%%%%%%%%%%%%%%%%%%%%%%%%%%%%%%%%%%%%%%%%%%%%%%%%%%
\section{Stability analysis of Algorithm 1}
\label{sec:stability}
\setcounter{equation}{0}

In this subsection, we consider the case $\mb{\bar h} = 0$ and homogeneous
Neumann boundary conditions that are dominant in numerical experiments.
We will show that the discrete exchange energy decreases in time under certain conditions
on the computational mesh for both explicit and implicit mimetic finite difference schemes.
More precisely, all matrices $\mbb{M}_{{\cal F},E}^{-1}$ must be M-matrices.
In addition, for the explicit scheme, the time step and the mesh size must satisfy a Courant condition.

Let $\|\cdot\|_{\cal F}$ and $\|\cdot\|_{{\cal F},E}$ denote the norms indexed by global and local inner products.
First, we formulate our main result and then prove it as well as two auxiliary lemmas.

\begin{theorem}
Let conditions of Lemmas~1 and 2 hold true. 
For the explicit scheme ($\theta=0$), we further assume the Courant condition
\begin{equation}\label{eq:courant}
  \Frac{k}{h^2} \leq {2\alpha \over C_1 (1+\alpha^2)}. 
\end{equation}
Then, for any time moment $t^j$, we have the following energy estimate:
$$
  || \bGRAD\,\mb{m}^{j+1} ||_{{\mathcal{F}}} \le 
  || \bGRAD\,\mb{m}^{j} ||_{{\mathcal{F}}}.
$$
\end{theorem}

\begin{proof}
Let $\hat{\mb{m}}^{j+1}$ denote the magnetization solution obtained on step 2a of Algorithm~1. 
Let $\mb{v}^j$ denote the change of magnetization, i.e. $\hat{\mb{m}}^{j+1} = \mb{m}^{j} + k \mb{v}^j$.
Note that ${\mb{m}}^{j}_E \cdot \mb{v}^j_E = 0$ and $|\mb{m}^{j}_E|=1$ for all $E$. 
Hence, $|\hat{\mb{m}}^{j}_E| \ge 1$ and Lemma~\ref{lemma:normdecreasing} gives us energy decrease
after normalization:
\begin{equation} \label{eq:dec}
  || \bGRAD\,\mb{m}^{j+1} ||_{{\mathcal{F}}} \le || \bGRAD\,\hat{\mb{m}}^{j+1} ||_{{\mathcal{F}}}.
\end{equation}
We have the following identity:
\begin{equation}\label{eq:energy-aux}
\begin{aligned}
  \| \bGRAD\big(\mb{m}^{j} + k \mb{v}^j\big)\|_{\mathcal{F}}^2 
  &= \|\bGRAD\,\mb{m}^{j}||^2_{\mathcal{F}} \\[1ex]
  &+ 2 k\, [\bGRAD\,\mb{m}^{j},\,\bGRAD\, \mb{v}^j ]_{\mathcal{F}}
   + k^2 \| \bGRAD\, \mb{v}^j\|_{\mathcal{F}}^2
\end{aligned}
\end{equation}
The second equation in the $\theta$-scheme \eqref{eq:dmdt} gives us the definition of $\mb{v}^j$:
$$
\begin{aligned} 
  \mb{v}_E^j 
  = -\mb{m}_E^j \times \bDIV_E\,\mb{p}_E^{j+\theta}
    -\alpha\, \big(\mb{m}_E^j \cdot \bDIV_E\,\mb{p}_E^{j+\theta} \big) \mb{m}_E^j
    +\alpha\, \bDIV_E\,\mb{p}_E^{j+\theta}
\end{aligned}
$$
Using this formula and the fact that $|\mb{m}_E^j|=1$, we can calculate the following quantity:
\begin{equation}\label{eq:lld1}
\begin{aligned} 
  \alpha \mb{v}_E^j + \mb{m}_E^j \times \mb{v}_E^j 
  = -(1+\alpha^2)\, \big((\mb{m}_E^j \cdot \bDIV_E\,\mb{p}_E^{j+\theta})\, \mb{m}_E^j
    -\bDIV_E\,\mb{p}_E^{j+\theta}\big).
\end{aligned}
\end{equation}
Taking the dot product of both sides with $\mb{v}_E^j$, summing up the results weighted 
with element volumes, and using the duality property of the mimetic operators, we obtain
\begin{equation}\label{eq:aux}
\begin{aligned}
  -\Frac{\alpha}{1+\alpha^2} [\mb{v}^j,\,\mb{v}^j]_{\mathcal{Q}} 
  &= -[\bDIV\,\mb{p}^{j+\theta},\, \mb{v}^j]_{\mathcal{Q}} 
   = [\bGRAD\big(\mb{m}^{j} + \theta k \mb{v}^j\big),\, \bGRAD \mb{v}^j]_{\mathcal{Q}} \\
  &= [\bGRAD\, \mb{m}^{j},\, \bGRAD \mb{v}^j]_{\mathcal{Q}} 
   + \theta k\,[\bGRAD \mb{v}^j,\, \bGRAD\,\mb{v}^j]_{\mathcal{Q}}
\end{aligned}
\end{equation}
Inserting \eqref{eq:aux} into \eqref{eq:energy-aux}, we have
$$
\begin{aligned}
  \| \bGRAD\, \hat{\mb{m}}^{j+1} \|_{\mathcal{F}}^2 
  = \| \bGRAD\, \mb{m}^{j} \|^2_{\mathcal{F}} 
   -\Frac{2\alpha\,k}{1+\alpha^2}\, \|\mb{v^j}\|_{\mathcal{Q}}^2 
   -k^2\, (2\theta - 1) \|\bGRAD\,\mb{v}^j\|_{\mathcal{F}}^2.
\end{aligned}
$$
This shows that for the implicit scheme, $\theta=1$, the energy decreases using equation \eqref{eq:dec}.
For the explicit scheme, we use Lemma~\ref{lemma:inverseestimate} to obtain
$$
  \| \bGRAD\, \mb{m}^{j+1} \|_{\mathcal{F}}^2 
  \le \| \bGRAD\, \mb{m}^{j} \|^2_{\mathcal{F}} 
   -k\,\left( \Frac{2\alpha}{1+\alpha^2} - \Frac{C_1\,k}{h^2}\right) \|\mb{v}^j\|_{\mathcal{Q}^h}^2 
$$
which shows that the energy decreases under the Courant condition \eqref{eq:courant}.
\end{proof}

What is left is to prove two technical lemmas used in Theorem~1.
The first lemma uses the stability condition of the mimetic scheme.
The second lemma assumes that the family of elemental matrices $\mbb{M}_{{\cal F},E}^{-1}$ 
contains M-matrices which imposes certain constraints on the shape of mesh cells.

\begin{lemma} \label{lemma:inverseestimate}
Let mesh $\Omega^h$ be shape regular and quasi-uniform. 
Then, for any $\mb{v}^h \in ({\cal Q}^h)^3$, we have the following 
inverse estimate:
\begin{equation} \label{eq:inv}
  \| \bGRAD\,\mb{v}^h\|_{\mathcal{F}}^2 \le \frac{C_1}{h^2} \| \mb{v}^h \|_{\mathcal{Q}}^2,
\end{equation}
where $C_1$ is a positive constant independent of $h$ and $\mb{v}^h$.
\end{lemma}

\begin{proof} For each component $u \in \{x,y,z\}$, we use definition of the 
mimetic operators, and stability condition \eqref{eq:stability} to obtain
$$
  \| \GRAD\, v^h_u \|_{\mathcal{F}} 
  = \|\mathbb{M}_{\mathcal{F}}^{-1/2}\, \DIV^T\, \mathbb{M}_{\mathcal{Q}} v^h_u \| 
  \le \Frac{C_2}{h^{d/2}} \|\DIV^T\, \mathbb{M}_{\mathcal{Q}} v^h_u \|
  \le \Frac{C_2}{h^{d/2}} \|\DIV^T\| \; \|\mathbb{M}_{\mathcal{Q}} v^h_u \|.
$$
where $C_2$ depends on $c_0$  in \eqref{eq:stability} and the shape regularity of mesh $\Omega^h$.
The definition of the divergence operator gives $\| \DIV^T \| \le C_3\, h^{-1}$, where
$C_3$ depends only on the mesh shape regularity constants.
The definition of the inner product matrix gives $\| \mbb{M}_{\mathcal{Q}}^{1/2} \| \le h^{d/2}$.
Thus, we have 
$$
  \| \GRAD\, v^h_u \|_{\mathcal{F}} 
  \le \Frac{C_2\,C_3}{h} \|v^h_u\|_{\mathcal{Q}}.
$$
The assertion of the lemma follows with $C_1 = (C_2\,C_3)^2$.
\end{proof}

\begin{lemma} \label{lemma:normdecreasing}
Let each elemental matrix $\mbb{M}_{{\cal F},E}$ satisfy two conditions: 
(a) $\mbb{M}_{{\cal F},E}^{-1}$ is an M-matrix, and 
(b) vector $\mbb{M}_{{\cal F},E}^{-1}\, \mbb{C}_E\, \bm{e}$ has positive entries.
Furthermore, let $\hat{\mb{v}}^h \in ({\cal Q}^h)^3$ be any vector and $\mb{v}^h$ be its normalization
such that $\mb{v}_E = \hat{\mb{v}}_E / |\hat{\mb{v}}_E|$ for all $E$ in $\Omega^h$.
Finally, let $|\hat{\mb{v}}_E| \ge 1$.
Then, we have energy decrease after the renormalization:
\begin{equation} \label{eq:normdec}
  \| \bGRAD\, \mb{v}^h \|_{\mathcal{F}}
  \le \| \bGRAD\,\hat{\mb{v}}^h \|_{\mathcal{F}}.
\end{equation}
\end{lemma}

\begin{proof}
Using the global and local mimetic gradient operators, we define vectors 
$$
\mb{q} = \bGRAD\,\mb{v}^h\quad\mbox{and}\quad 
\mb{q}_E = \bGRAD_E\,
  \begin{pmatrix}\mb{v}_E\\ \widetilde{\mb{v}}_E \end{pmatrix}.
$$
The global and local gradient operators are equivalent when $q_{u,E_1,f} + q_{u,E_2,f} = 0$
on each internal mesh edge shared by two elements $E_1$ and $E_2$.
Under this continuity condition, the additivity property of the inner product gives:
$$
  \| \mb{q} \|_{\mathcal{F}}^2
  = \sum\limits_{E \in \Omega^h} \| \mb{q}_E \|_{\mathcal{F},E}^2.
$$

Formula \eqref{eq:pW} for the local mimetic gradient operator and definition of
diagonal matrix $\mbb{C}_E$ give
\begin{equation}\label{eq:quE}
  \| \mb{q}_{u,E} \|_{\mathcal{F},E}^2 
  = \begin{pmatrix} v_{u,E} \\[1ex] \widetilde{v}_{u,E} \end{pmatrix}^T
  \begin{pmatrix}
    \bm{e}^T \mbb{T}_E\, \bm{e} & -\bm{e}^T \mbb{T}_E \\[1ex]
    -\mbb{T}_E\, \mb{e}         & ~\mbb{T}_E 
  \end{pmatrix}
  \begin{pmatrix} v_{u,E} \\[1ex] \widetilde{v}_{u,E} \end{pmatrix},
  \qquad 
  \mbb{T}_E = \mbb{C}_E\, \mbb{M}_{{\cal F},E}^{-1}\, \mbb{C}_E.
\end{equation}
The entries of vector $\widetilde{v}_{u,E}^h$ are associated with mesh edges
and defined completely by the flux continuity condition $q_{u,E_1,f} + q_{u,E_2,f} = 0$.
To find another form of this condition, we sum up equations \eqref{eq:quE} which gives
\begin{equation}\label{eq:assembled}
  \| \mb{q}_u \|_{\mathcal{F}}^2
  = \sum\limits_{E \in \Omega^h} \| \mb{q}_{u, E} \|_{\mathcal{F},E}^2
  = \begin{pmatrix} v_u^h \\[1ex] \widetilde{v}_u^h \end{pmatrix}^T
  \begin{pmatrix}
    \mbb{T}^{EE} & \mbb{T}^{Ef} \\[1ex]
    \mbb{T}^{fE} & \mbb{T}^{ff}
  \end{pmatrix}
  \begin{pmatrix} v_u^h \\[1ex] \widetilde{v}_u^h \end{pmatrix}.
\end{equation}
Now, inserting \eqref{eq:pW} in the flux continuity equations pre-multiplied by $|f|$, we obtain
$$
  \mbb{T}^{fE}\, v_u^h + \mbb{T}^{ff} \widetilde{v}_u^h = 0.
$$
Using this relationship in formula \eqref{eq:assembled}, we obtain
$$
  \| \mb{q}_u \|_{\mathcal{F}}^2 = 
    (v_u^h)^T \big(\mbb{T}^{EE} - \mbb{T}^{Ef}\, (\mbb{T}^{ff})^{-1}\, \mbb{T}^{fE}\big)\, v_u^h 
  =: (v_u^h)^T\, \mbb{S}^{EE}\, v_u^h.
$$
The Schur complement $\mbb{S}^{EE}$ has one important property.
According to \cite{lipnikov2011analysis}, the conditions (a) and (b) 
imply that the local matrices in \eqref{eq:quE} are singular irreducible M-matrices
with the single null vector $\bm{e}$. 
Hence, the 
assembled matrix in \eqref{eq:assembled} is a singular M-matrix.
From linear algebra we know that the Schur complement is also a singular M-matrix.
Since $\mbb{S}^{EE}\, \bm{e} = 0$, the following vector-matrix-vector product can be broken into 
assembly of $2\times2$ matrices,
\begin{equation}\label{eq:SEE}
  (v_u^h)^T\, \mbb{S}^{EE}\, v_u^h
  = \sum\limits_{i < j} 
    \beta_{ij}
    \begin{pmatrix} v_{u,E_i} \\ v_{u,E_j} \end{pmatrix}^T
    \begin{pmatrix}
     ~1 & -1 \\
     -1 & ~1 
    \end{pmatrix}
    \begin{pmatrix} v_{u,E_i} \\ v_{u,E_j} \end{pmatrix},
\end{equation}
with non-negative weights $\beta_{ij}$.
Recall that $v_{u,E_i} = \hat{v}_{u,E_i} / |\hat{\mb{v}}_{E_i}|$
and $|\hat{\mb{v}}_{E_i}|\ge 1$.
Thus, we have the following estimate: 
$$
\begin{aligned}
\sum_{u \in \{x,y,z\}} &\left( \Frac{(\hat{v}_{u,E_i})^2}{|\hat{\mb{v}}_{E_i}|^2} 
    -2 \Frac{\hat{v}_{u,E_i}}{|\hat{\mb{v}}_{E_i}|}\,
       \Frac{\hat{v}_{u,E_j}}{|\hat{\mb{v}}_{E_j}|}
    + \Frac{(\hat{v}_{u,E_j})^2}{|\hat{\mb{v}}_{E_j}|^2} \right) 
  = 2 - 2 \sum_{u \in \{x,y,z\}} \Frac{\hat{v}_{u,E_i}}{|\hat{\mb{v}}_{E_i}|}\,
           \Frac{\hat{v}_{u,E_j}}{|\hat{\mb{v}}_{E_j}|}\\
  & \leq |\hat{\mb{v}}_{E_i}|\, |\hat{\mb{v}}_{E_j}| 
    \left( \frac{|\hat{\mb{v}}_{E_i}|}{|\hat{\mb{v}}_{E_j}|} 
          +\frac{|\hat{\mb{v}}_{E_j}|}{|\hat{\mb{v}}_{E_i}|} 
          - 2 \sum_{u \in \{x,y,z\}} \Frac{\hat{v}_{u,E_i}}{|\hat{\mb{v}}_{E_i}|}\,
              \Frac{\hat{v}_{u,E_j}}{|\hat{\mb{v}}_{E_j}|}\right)
  = \sum_{u \in \{x,y,z\}} (\hat{v}_{u,E_i}-\hat{v}_{u,E_j})^2.
\end{aligned}
$$
We conclude that
$$
 \sum_{u \in \{x,y,z\}} (v_u^h)^T\, \mbb{S}^{EE}\, v_u^h \le
\sum_{u \in \{x,y,z\}}  (\hat{v}_u^h)^T\, \mbb{S}^{EE}\, \hat{v}_u^h
  = \sum_{u \in \{x,y,z\}}\| \GRAD\, \hat{v}_u^h \|_{{\cal F}}^2.
$$
and the assertion of the lemma follows.
\end{proof}

\begin{remark}
To comply with the conditions of Lemma~\ref{lemma:normdecreasing}, we have
to select special matrices $\tilde{\mathbb{G}}$ in \eqref{eq:WFE}. 
A simple optimization algorithm for this task is proposed in \cite{Lipnikov:FVCA}.
\end{remark}

%%%%%%%%%%%%%%%%%%%%%%%%%%%%%%%%%%%%%%%%%%%%%%%%%%%%%%%%%%%%%%%%%%%%%
\section{Numerical Examples}
\label{sec:numerical}
\setcounter{equation}{0}

\subsection{Explicit time integration scheme ($\theta=0$)}

In this subsection, we consider the explicit time integration scheme, i.e. $\theta=0$ in \eqref{eq:dmdt}.
Analytical solutions for the LL equation are available for special forms of the effective 
field (\ref{eq:LLenergy}), for instance when it has only the exchange energy term, $\mb{h}= \Delta \mb{m}$.
We consider the analytical solution from \cite{fuwa2012finite} with the periodic boundary conditions, namely
\begin{equation}\label{eq:exact}
\begin{aligned}
  &m_x(x_1,x_2,t) =\frac{1}{d(t)} \sin \beta \cos(\kappa(x_1+x_2)+g(t)), \\
  &m_y(x_1,x_2,t) =\frac{1}{d(t)} \sin \beta \sin(\kappa(x_1+x_2)+g(t)), \\
  &m_z(x_1,x_2,t) =\frac{1}{d(t)} e^{2 \kappa^2 \alpha t} \cos \beta.
\end{aligned}
\end{equation}
where $\beta = \frac{\pi}{12}$, $\kappa=2\pi$, $d(t) = \sqrt{\sin^2\beta + e^{4\kappa^2\alpha t} \cos^2 \beta}$ 
and $g(t)=\frac{1}{\alpha}\log(\frac{d(t)+e^{2\kappa^2 \alpha t} \cos \beta}{1+\cos \beta})$.
Note that $m_z \to 1$ as $t \to \infty$. 
The snapshot of the analytical solution (\ref{eq:exact}) at time $0$ is in Fig.~\ref{figure:analytical}. 
\begin{figure}[h]
\begin{center}
  \includegraphics[width=.4\linewidth,trim=20 25 20 20,clip]{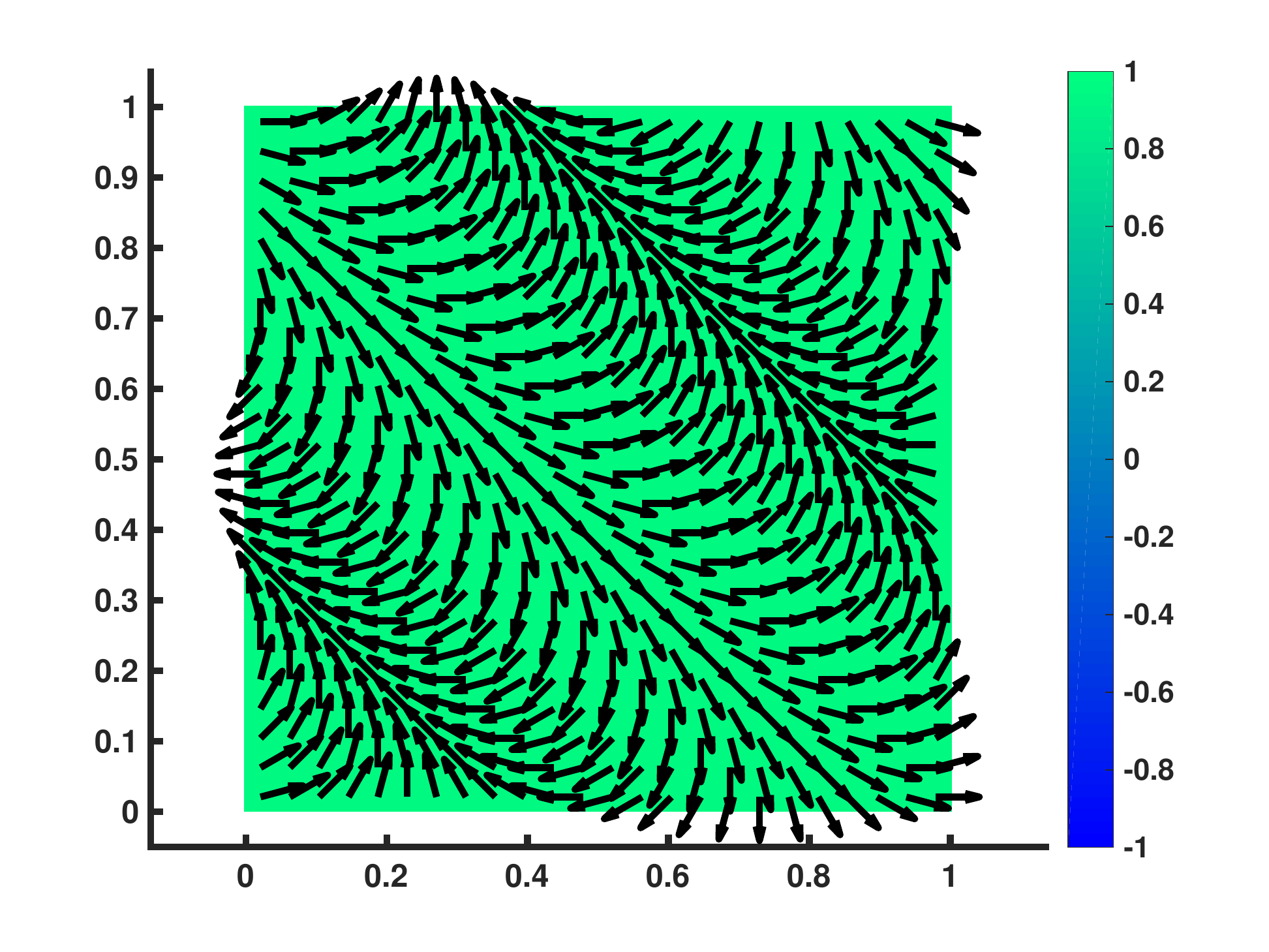}
  \caption{ Analytical solution (\ref{eq:exact}) at time $0$. The vectors in the plot denote the $m_x$ and $m_y$ components, and the color denotes the $m_z$ component. }
  \label{figure:analytical}
\end{center}
\end{figure}
We perform simulation on the time interval $(0, T)$, where $T=0.001$.

Let us consider a square mesh with mesh step $h$ occupying the unit square $\Omega$
and set the time step $k = 8 \times 10^{-7} h^2$.  
We measure errors in the magnetization and its flux in the mesh dependent $L^2$-type norms 
$\|\cdot\|_{\cal Q}$ and $\|\cdot\|_{\cal F}$, respectively.
In addition to that, we consider the maximum norm
$$
  \| \mb{m^h}- \mb{m}^I \|_{L^\infty}
  = \max\limits_{E \in \Omega^h} | \mb{m}_E - \mb{m}^I_E |,
$$
where $\mb{m}^I \in ({\cal Q}^h)^3$ is a projection of the analytical solution on the
discrete space.
For this projection, we simply take the value of $\mb{m}$ at centroids of elements $E$.
The projection $\mb{p}^I$ is defined in a similar way.
Table~\ref{table:mimetic} shows convergence rates in different norms.
Observe that the explicit scheme leads to the second-order convergence for both 
the magnetization and its flux.

\begin{table}[h]
\centering
\begin{tabular}{c| c| c|c}
  $1/h$  &  $\| \mb{m}^h - \mb{m}^I\|_{L^\infty}$ & 
            $\| \mb{m}^h - \mb{m}^I\|_{\cal Q}$ & 
            $\| \mb{p}^h - \mb{p}^I\|_{\cal F}$ \\
  \hline
   32  &   8.222e-05  &  8.360e-05    &	2.967e-03 \\
   64  &   2.060e-05  &  2.092e-05    &  7.418e-04 \\
  128  &   5.154e-06  &  5.231e-06   &  1.854e-04   \\
  256  &   1.289e-06  &  1.308e-06	  &  4.636e-05  \\
  \hline
  rate &  2.00         &  2.00   &	2.00  \\
\end{tabular}
\caption{Convergence analysis of the explicit time integration scheme on uniform square meshes.}
\label{table:mimetic}
\end{table}

%\begin{figure}[h]
%\begin{center}
%  \includegraphics[width=.48\linewidth,trim=5 0 40 20,clip]{mimetic}
%  \includegraphics[width=.48\linewidth,trim=5 0 40 20,clip]{mimeticflux}
%  \label{figure:mimetic}
%  \caption{Error plot of $\left\Vert \mb{m^h}- \mb{m} \right\Vert_{L^\infty}$ 
%  and $\left\Vert \mb{m^h}- \mb{m} \right\Vert_{L^2}$  (left) and
%  $\left\Vert \mb{p^h}- \mb{p} \right\Vert_{L^2}$ (right) with respect to the 
% mesh size $h$ using explicit time integration scheme on uniform square mesh. }
%\end{center}
%\end{figure}

%%%%%%%%%%%%%%%%%%%%%%%%%%%%%%%%%%%%%%%%%%%%%%%%%%%%%%%%%%%%%%%%%%%%%
\subsection{Implicit time integration scheme ($\theta = 1$)}
\subsubsection{Uniform square meshes}
\label{subsec:square}

In this subsection, we consider the implicit time integration scheme, i.e. $\theta=1$ in \eqref{eq:dmdt}.
The analytical solution is given by (\ref{eq:exact}). 
We set the time step $k = 0.008\, h^2$ so that the first-order time integration error
will not affect our conclusions.
Convergence rates shown in Table~\ref{table:mixedstructrued} indicate
the second-order convergence for the magnetization and the first-order convergence for its flux.
A rigorous convergence analysis of Algorithm~1, which is beyond the scope of this work, 
is required to explain lack of flux super-convergence in this scheme.

\begin{table}[h]
\centering
\begin{tabular}{c| c| c|c}
  $1/h$  &  $\| \mb{m}^h - \mb{m}^I\|_{L^\infty}$ & 
            $\| \mb{m}^h - \mb{m}^I\|_{\cal Q}$ & 
            $\| \mb{p}^h - \mb{p}^I\|_{\cal F}$ \\
  \hline
   32 &   9.082e-05   &  9.195e-05  	&	2.531e-02 \\
   64 &   2.273e-05   &  2.302e-05   	&	1.261e-02 \\
  128 &   5.687e-06   &  5.756e-06   	&	6.301e-03 \\
  256 &   1.422e-06   &  1.439e-06   	&	3.150e-03 \\
  \hline
  rate & 2.00          &  2.00  			&1.00 \\
\end{tabular}
\caption{Convergence analysis of the implicit time integration scheme on uniform square meshes.}
\label{table:mixedstructrued}
\end{table}

%\begin{figure}[h]
%\begin{center}
%  \includegraphics[width=.48\linewidth,trim=5 0 40 20,clip]{mixedstructured}
%  \includegraphics[width=.48\linewidth,trim=5 0 40 20,clip]{mixedstructuredflux}
%  \label{figure:mixedstructured}
%  \caption{Error plot of $\left\Vert \mb{m^h}- \mb{m} \right\Vert_{L^\infty}$ and 
%  $\left\Vert \mb{m^h}- \mb{m} \right\Vert_{L^2}$  (left) and
%  $\left\Vert \mb{p^h}- \mb{p} \right\Vert_{L^2}$ (right) with respect to the mesh 
%  size $h$ using Algorithm 1 on uniform square mesh.}
%\end{center}
%\end{figure}

%%%%%%%%%%%%%%%%%%%%%%%%%%%%%%%%%%%%%%%%%%%%%%%%%%%%%%%%%%%%%%%%%%%%%
\subsubsection{Smoothly distorted and randomized quadrilateral meshes}

In this subsection, we continue convergence analysis of the implicit time integration scheme
using the analytical solution described above.
This time, we consider the randomized and smoothly distorted meshes shown in Fig.~\ref{figure:distorted}. 
The randomized mesh is built from the uniform square mesh by a random distortion of its nodes:
\begin{equation}
  x := x+ 0.2\,\xi_x\, h, \quad y := x+ 0.2\,\xi_y\,h,
\end{equation}
where $\xi_x$ and $\xi_y$ are random variables between $-1$ and $1$.
The mesh perturbation was modified on the boundary so that periodic boundary conditions could be used.

The smoothly distorted mesh is built from the uniform square mesh using a smooth map to calculate new 
positions of mesh nodes:
\begin{equation}
\begin{aligned}
  x := x + 0.1 \sin (2\pi x) \sin(2\pi y), \\
  y := y + 0.1 \sin (2\pi x) \sin(2\pi y).
\end{aligned}
\end{equation}

\begin{figure}[h]
\begin{center}
  \includegraphics[width=.4\linewidth,trim=80 35 50 20,clip]{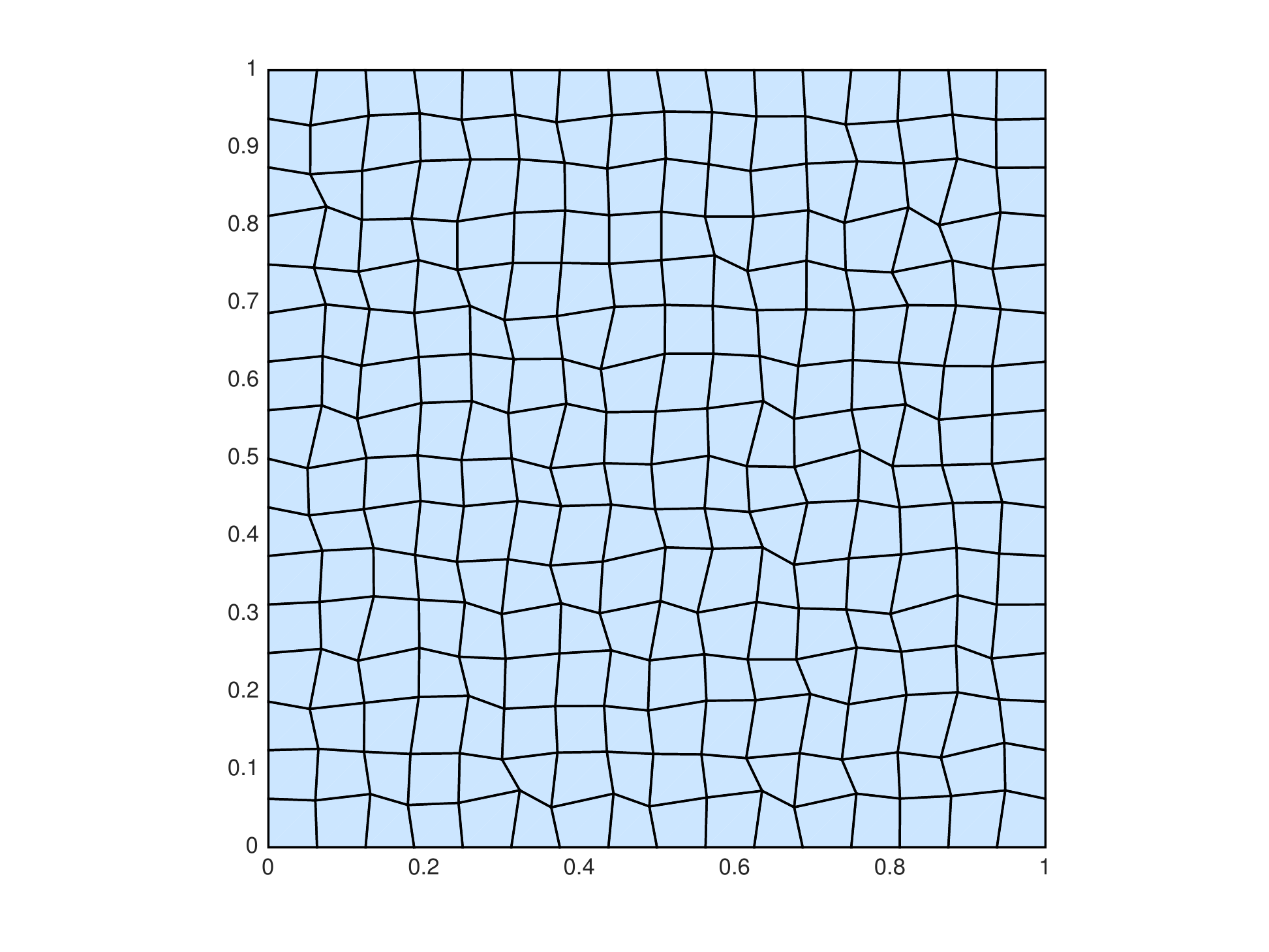}
  \includegraphics[width=.4\linewidth,trim=80 35 50 20,clip]{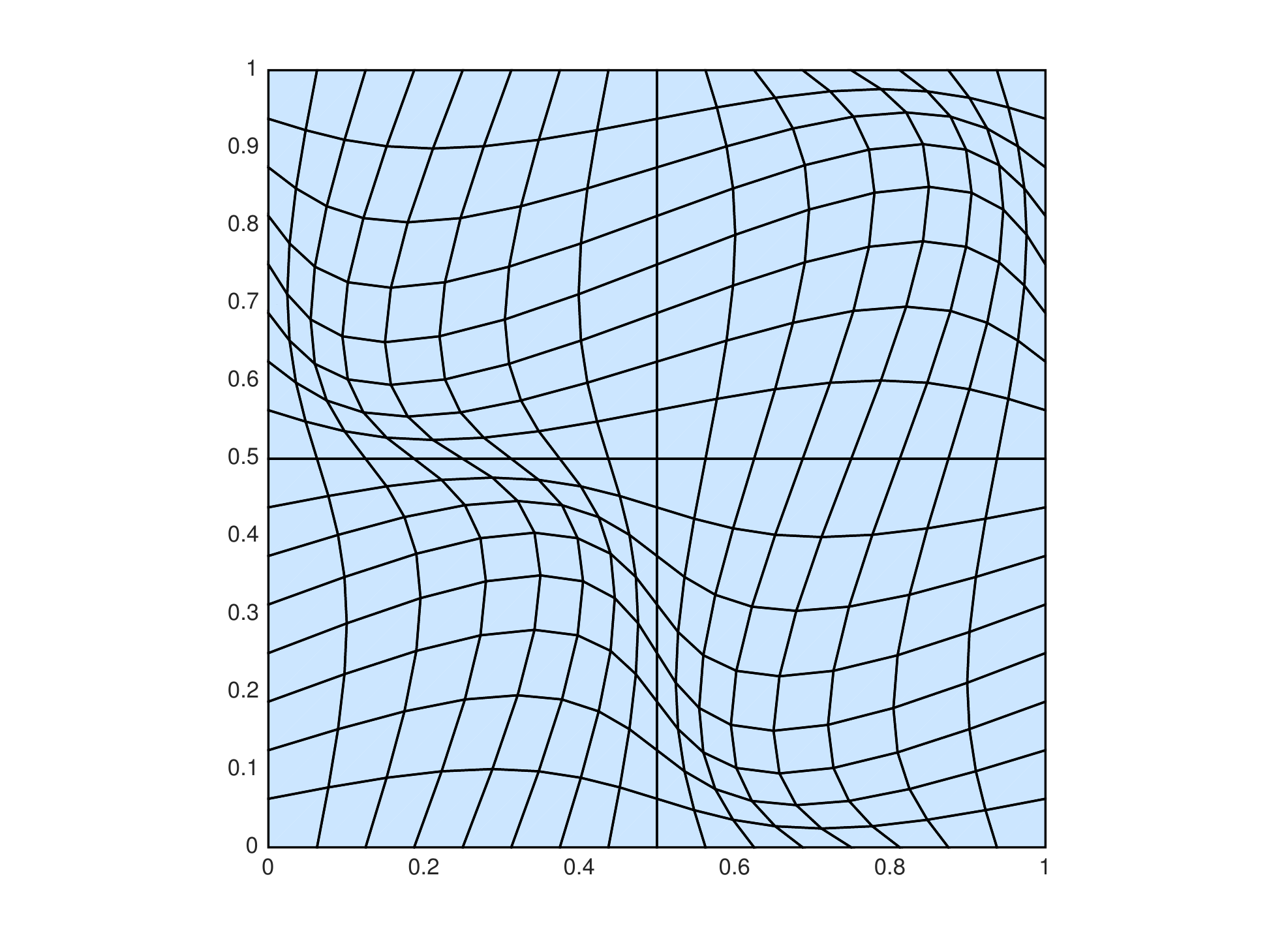}
  \caption{Modified meshes: randomized mesh (left), and smoothly distorted mesh (right).}
  \label{figure:distorted}
\end{center}
\end{figure}
We set the time step $k = 0.008\, h^2$. 
The errors are summarized in Table~\ref{table:mixedrandom} and Fig.~\ref{figure:convrandom}.
Again, we observe the second-order convergence rate for the magnetization 
and the first-order convergence rate for its flux.

\begin{table}[h]
\centering
\begin{tabular}{c| c| c|c||c|c|c}
  & \multicolumn{3}{|c||}{Randomized mesh} & \multicolumn{3}{|c}{Smoothly distorted mesh}\\  
  \hline
  $1/h$ & $\| \mb{m}^h - \mb{m}^I\|_{L^\infty}$ & $\| \mb{m}^h - \mb{m}^I\|_{\cal Q}$ & 
          $\| \mb{p}^h - \mb{p}^I\|_{\cal F}$ &  
          $\| \mb{m}^h - \mb{m}^I\|_{L^\infty}$ & $\| \mb{m}^h - \mb{m}^I\|_{\cal Q}$ & 
          $\| \mb{p}^h - \mb{p}^I\|_{\cal F}$ \\ 
  \hline
   16  &  3.121e-03  &  9.460e-04  & 6.560e-02 &  1.751e-03  &  1.009e-03 & 7.611e-02 \\
   32  &  1.005e-03  &  2.988e-04  & 3.249e-02 &  5.477e-04  &  2.809e-04 & 3.000e-02 \\
   64  &  2.585e-04  &  7.258e-05  & 1.615e-02 &  1.432e-04  &  7.214e-05 & 1.362e-02 \\
  128  &  7.939e-05  &  1.807e-05  & 8.135e-03 &  3.608e-05  &  1.816e-05 & 6.623e-03 \\
  \hline
  rate & 1.79         &  1.92   & 1.00       &  1.87         &  1.94  & 1.17\\
\end{tabular} 
\caption{Convergence analysis of the implicit time integration scheme on distorted meshes.}
\label{table:mixedrandom}
\end{table}

\begin{figure}[h]
\begin{center}
  \includegraphics[width=.49\linewidth,trim=5 -10 20 5,clip]{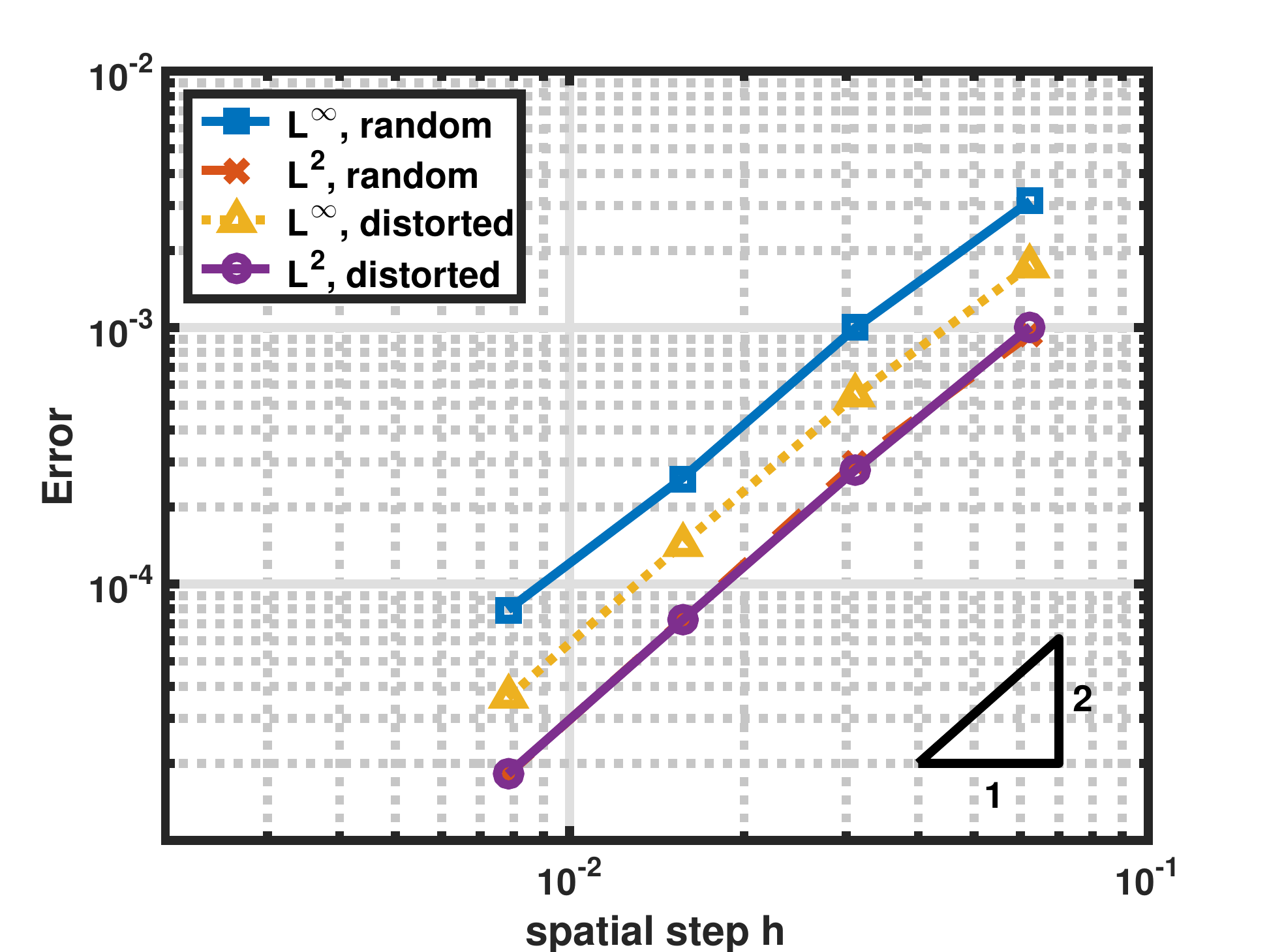}
  \includegraphics[width=.49\linewidth,trim=5 -10 20 5,clip]{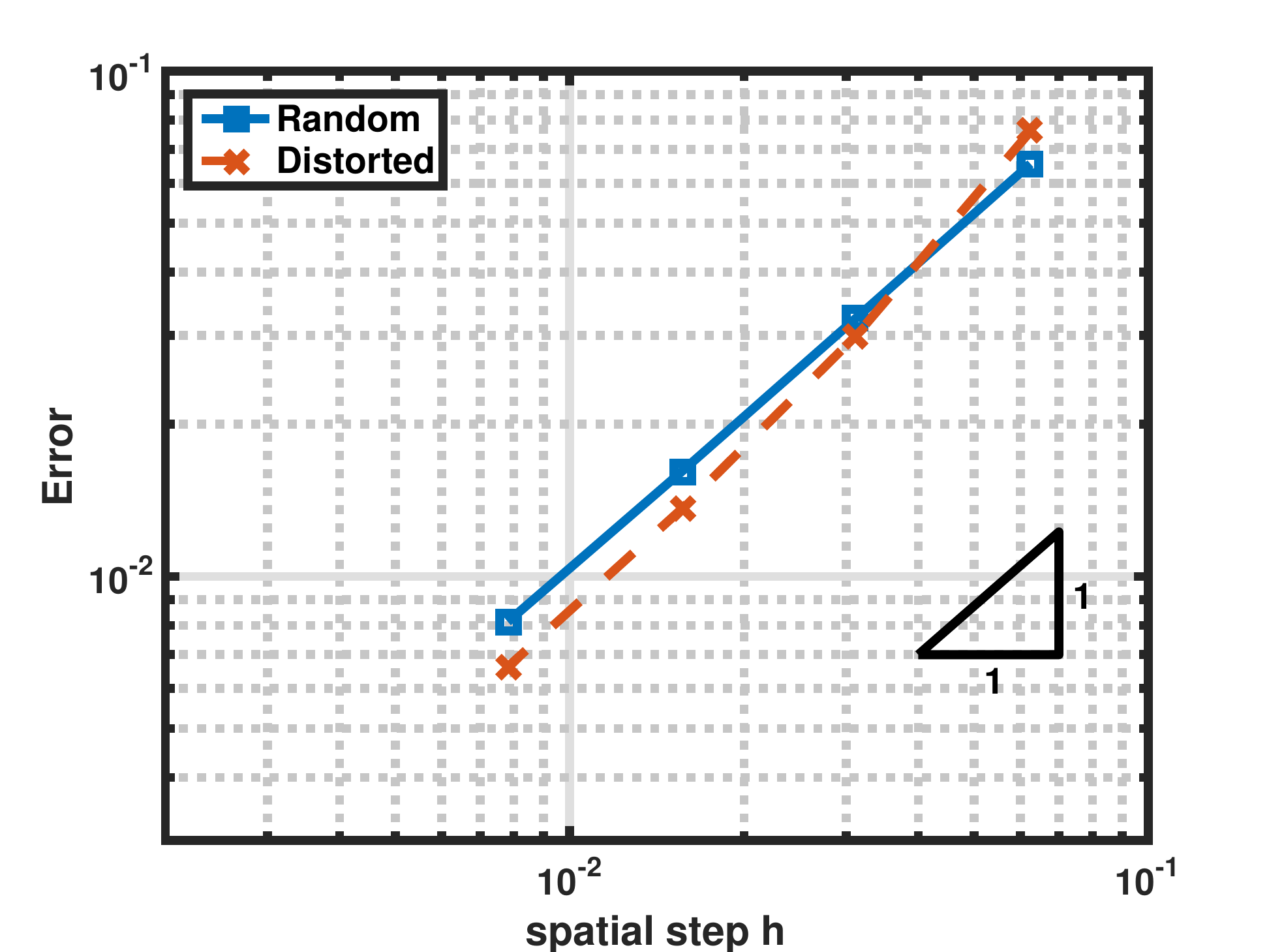}
  \caption{Error plot of $\| \mb{m}^h - \mb{m}^I \|_{L^\infty}$ and $\| \mb{m}^h - \mb{m}^I \|_{\cal Q}$ (left) 
  and $\| \mb{p}^h - \mb{p}^I \|_{\cal F}$ (right) with respect to the mesh size $h$ on 
  the randomized and smoothly distorted meshes shown in Fig.~\ref{figure:distorted}. }
  \label{figure:convrandom}
\end{center}
\end{figure}

In these experiments we used formula \eqref{eq:WFE} with constant $\tilde\gamma$ defined by 
the scaled trace of the first term, so that $\mathbb{M}_{{\cal F},E}^{-1}$ is not always an M-matrix.
Therefore, in Fig.~\ref{figure:energyandom}, we plot the exchange energy $\frac12 [\mb{p}^h,\, \mb{p}^h]_{\cal F}$ 
as the function of time for two different mesh resolutions $h=\frac{1}{32}$ and $h=\frac{1}{64}$
and for both randomized and smoothly distorted meshes. 
The figure shows monotone decrease of the exchange energy in time. 
These results suggest that the M-matrix conditions are sufficient but may not be necessary.

Furthermore, we study the impact of the constant $\tilde\gamma$  in \eqref{eq:WFE} on the solution accuracy.
Fig.~\ref{figure:gamma} shows errors as functions of $\gamma_0$ on a randomly distorted mesh with mesh 
size $h=1/32$, where $ \tilde \gamma  = \gamma_0 \frac{1}{|E|} \text{ trace} (\mathbb{N}\, \mathbb{N}^T)$.
With only one free parameter, we are not able to minimize errors in both, the magnetization and its flux,
and the full matrix of parameter has to be used instead.
Also, note that there is much room for varying $\tilde \gamma$ with relatively small increase of the error.

\begin{figure}[h]
\begin{center}
  \includegraphics[width=.6\linewidth,trim=0 0 0 0,,clip]{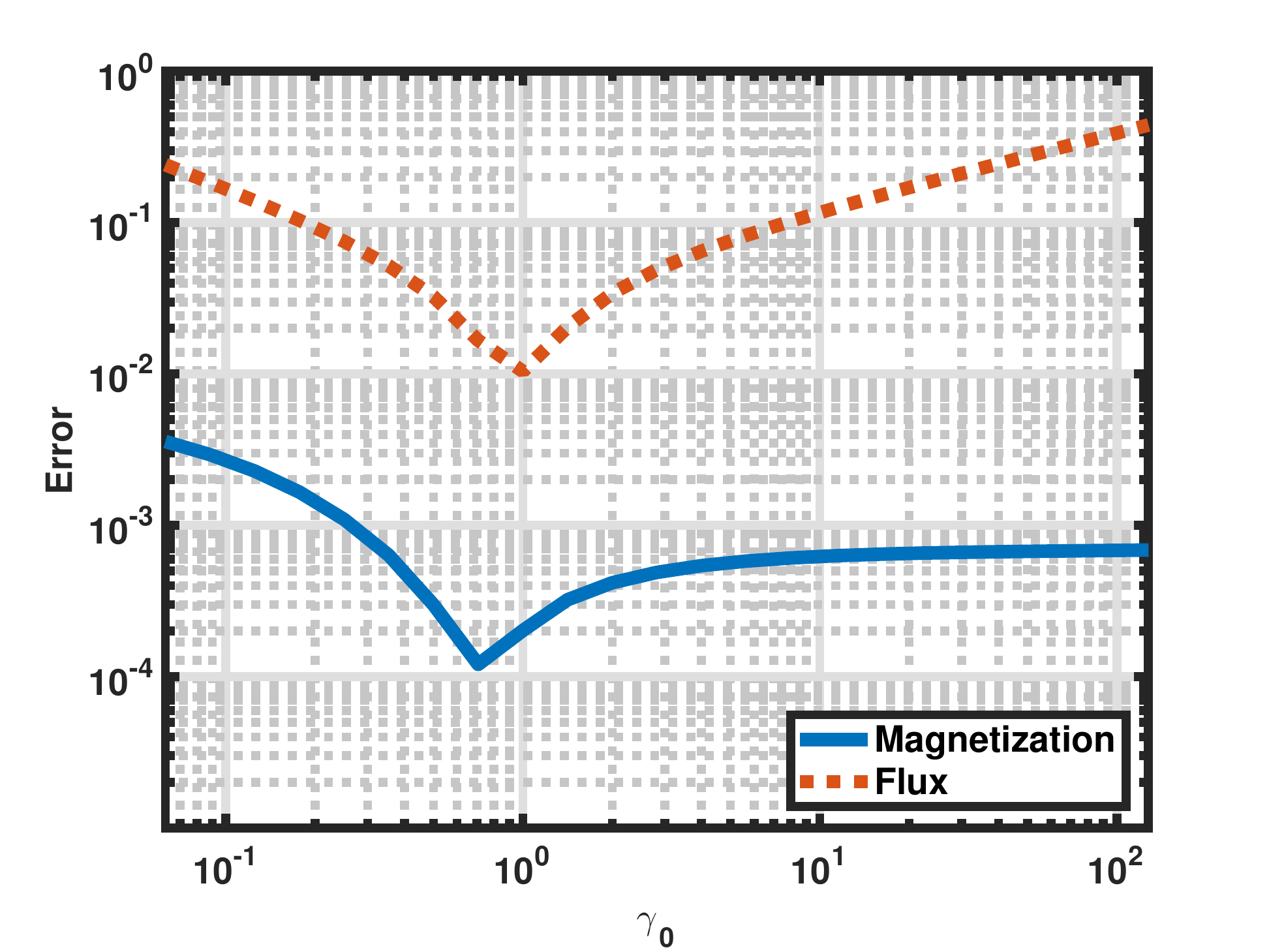}
\caption{Errors as function of $\gamma_0$ on a randomly distorted mesh with mesh size $h=1/32$ 
         (Fig.~\ref{figure:distorted} left)  }
\label{figure:gamma}
\end{center}
\end{figure}

\begin{figure}[h]
\begin{center}
  \includegraphics[width=.49\linewidth,trim=5 0 40 10,,clip]{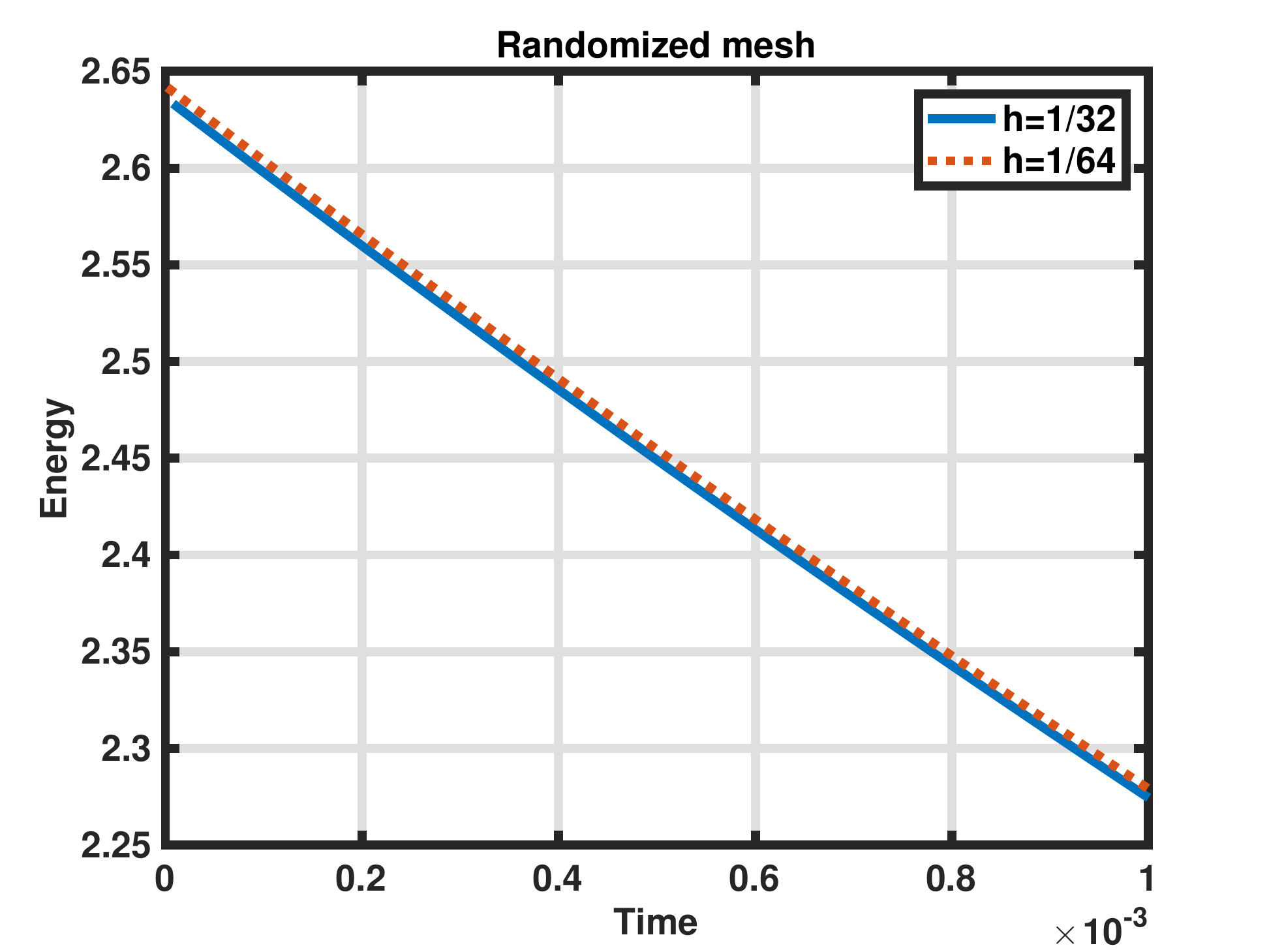}
  \includegraphics[width=.49\linewidth,trim=5 0 40 10,,clip]{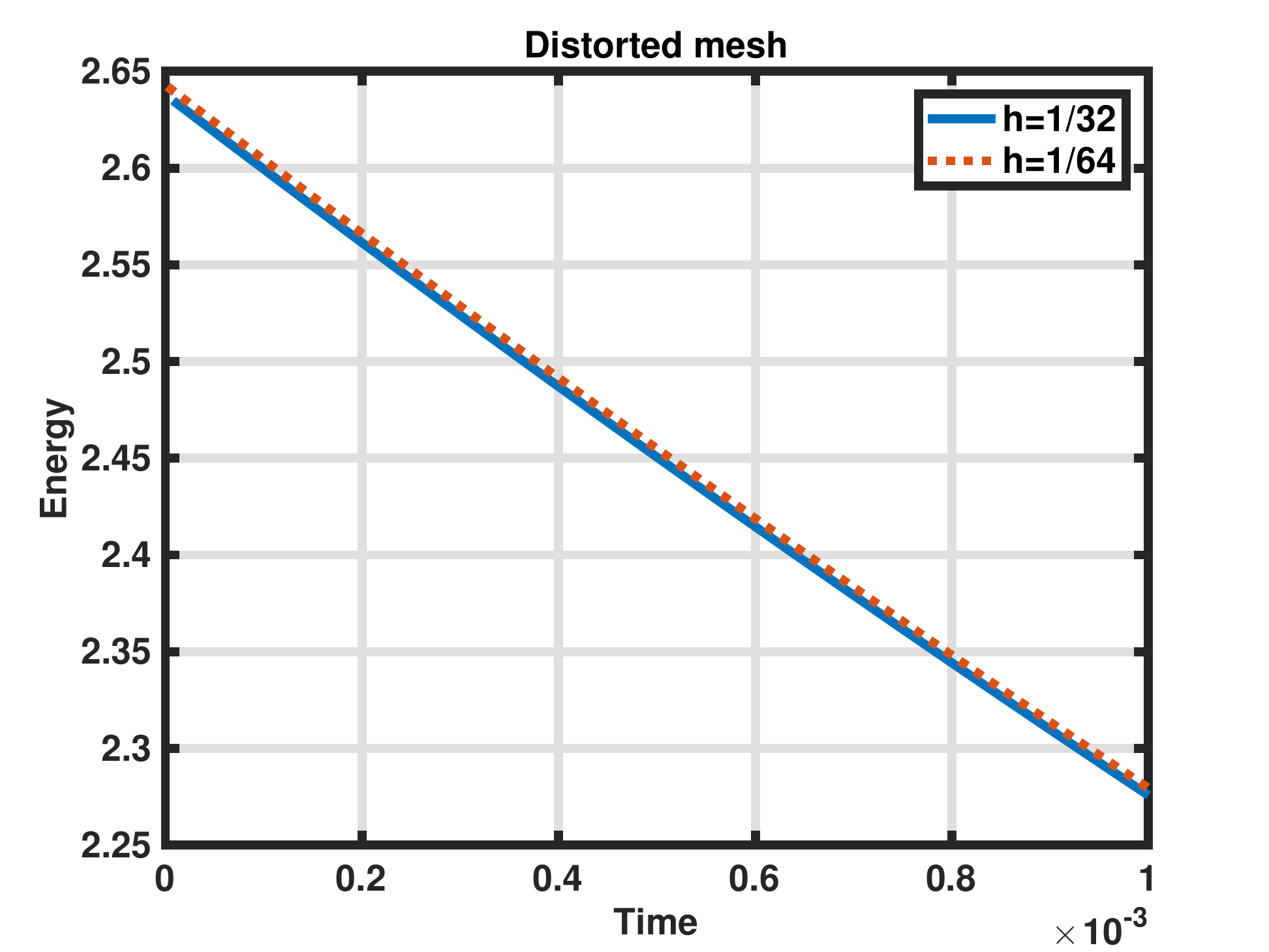}
\caption{Decrease of the exchange energy on randomized (left) and smoothly distorted (right) meshes.}
\label{figure:energyandom}
\end{center}
\end{figure}

%%%%%%%%%%%%%%%%%%%%%%%%%%%%%%%%%%%%%%%%%%%%%%%%%%%%%%%%%%%%%%%%%%%%%
\subsubsection{Convergence analysis for problems with the Dirichlet boundary condition}

In this subsection, we conduct numerical experiments for the LL equation 
with the effective field (\ref{eq:LLenergy}) $\bm{h}= \Delta \mb{m}$ using Algorithm~1
and the implicit time integration scheme.
We consider the analytical solution (\ref{eq:exact}) but now with the Dirichlet boundary condition.
The Dirichlet boundary condition allows us to consider more general domains as the circular
domain with center $(0.5, 0.5)$ and radius $0.5$ shown on the left panel in Fig.~\ref{figure:mixedstructured}.
This figure shows also a logically square mesh fitted to the domain.
This mesh has four elements (corresponding to four corners of the original square mesh) that
are almost triangles. 
However, for the mimetic framework, such elements are classified as shape regular elements 
(see \cite{brezzi2005family} for more detail) and do not alter the convergence rates.

\begin{figure}[h]
\begin{center}
  \includegraphics[width=.40\linewidth,trim=60 30 50 20,clip]{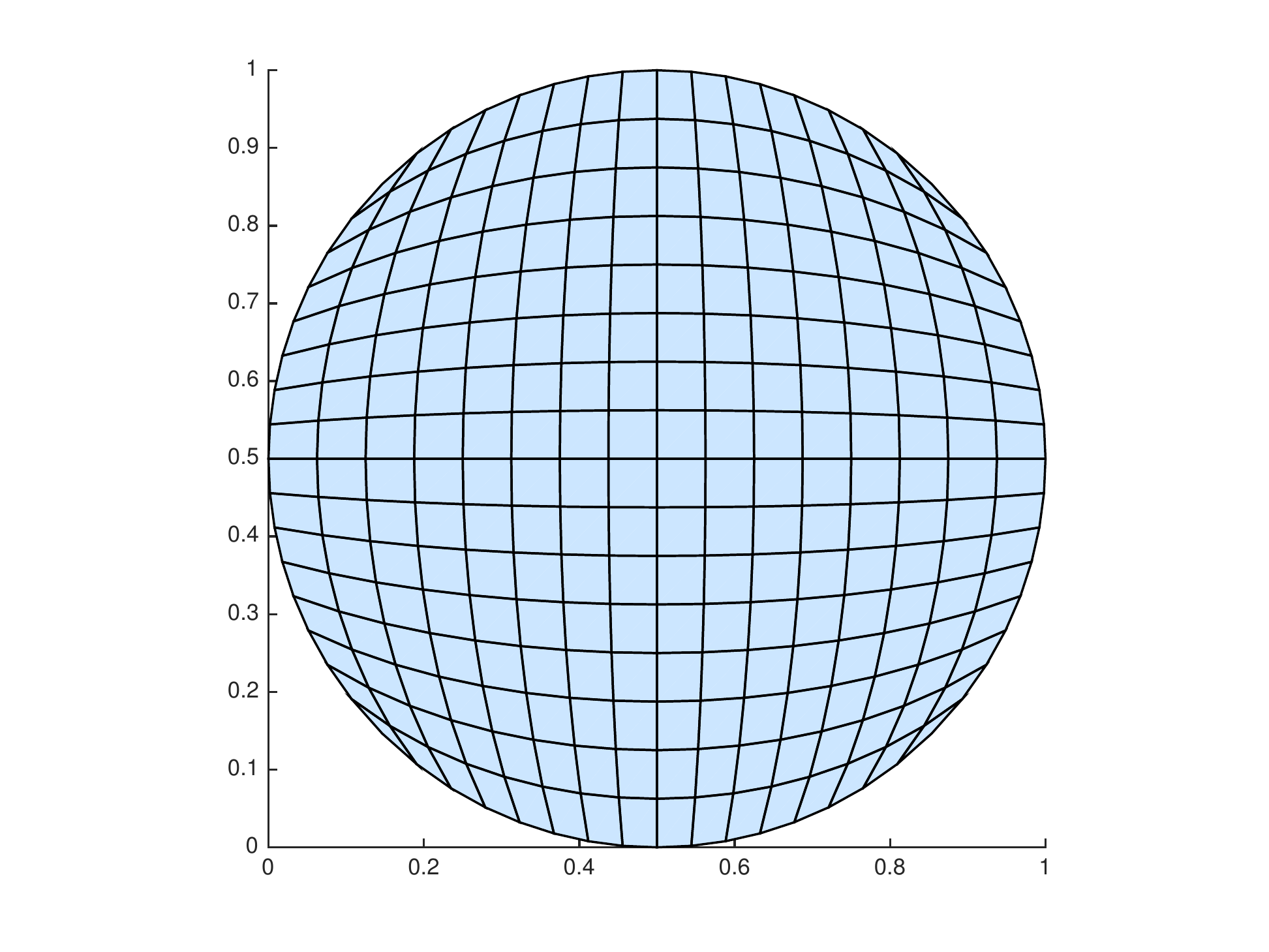}\quad
  \includegraphics[width=.40\linewidth,trim=60 30 50 20,clip]{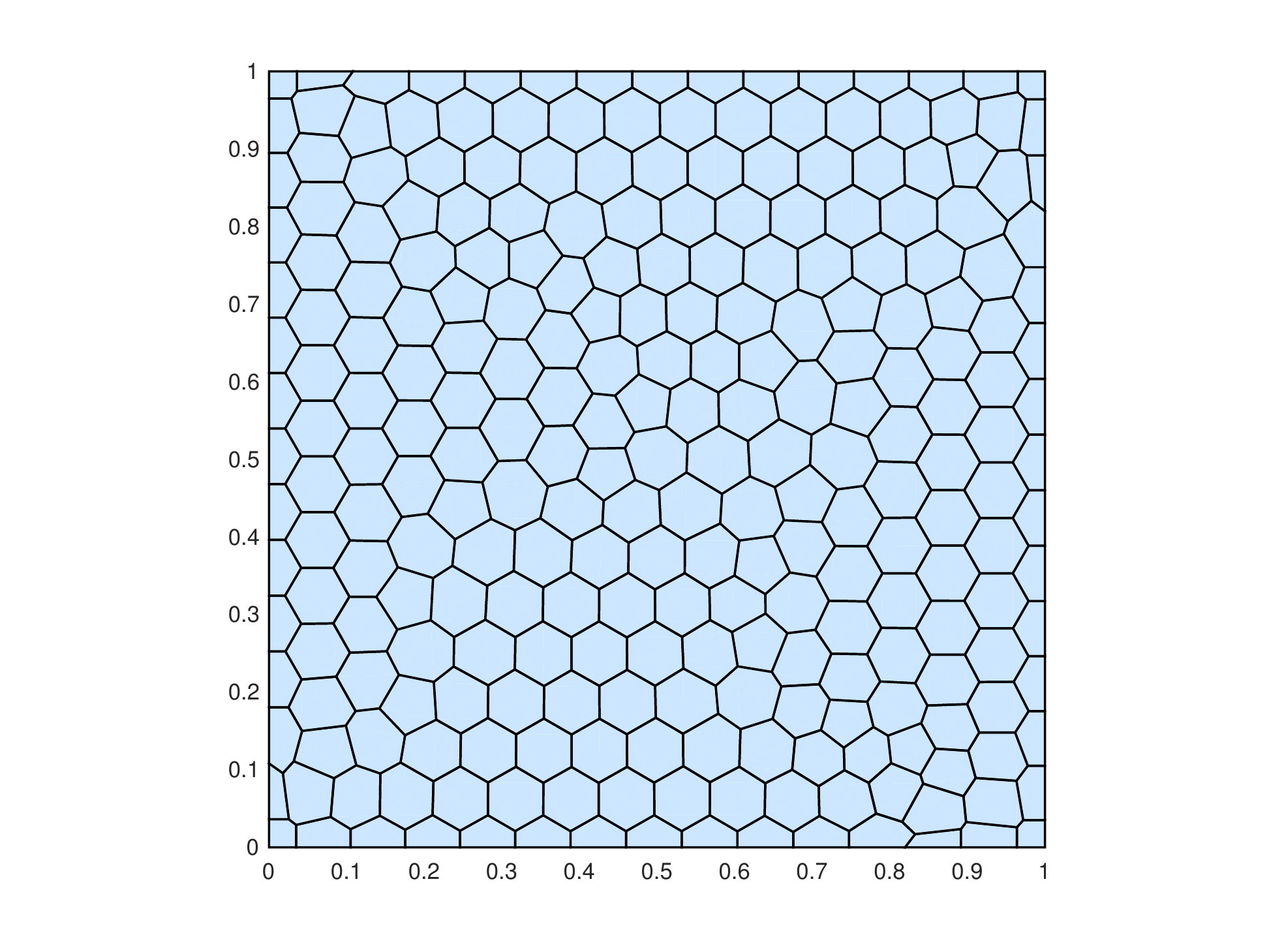}
\caption{Logically square mesh fitted to the circular domain (left panel) and polygonal mesh (right panel).}
\label{figure:mixedstructured}
\end{center}
\end{figure}

The computed errors are summarized in Table~\ref{table:dirichlet} and Fig.~\ref{figure:dirichlet},
where we present convergence results for the unit square, circular domain with various meshes. 
In addition to distorted meshes described in the previous subsection, we 
conduct numerical experiments on a sequence of polygonal meshes as in Fig.~\ref{figure:mixedstructured}.
The time step is $k = 0.008 \,h^2$ for all meshes but polygonal ones where it is set to $k = 0.004 \,h^2$.
We observe again the second-order convergence rate for the magnetization and the first-order 
convergence rate for its flux.

\begin{table}[h]
\centering
\begin{tabular}{c|c|c|c|c|c|c}
  & \multicolumn{6}{c}{Uniform square meshes}  \\  
  \hline
  $1/h$ & $\| \mb{m}^h - \mb{m}^I\|_{L^\infty}$ & ratio & 
          $\| \mb{m}^h - \mb{m}^I\|_{\cal Q}$   & ratio &
          $\| \mb{p}^h - \mb{p}^I\|_{\cal F}$   & ratio  \\
  \hline
   8   &   6.799e-03 & 0.83&  3.239e-03  & 0.83  &1.585e-01 	& 1.40 \\
  16   &   3.812e-03 & 1.81&  1.347e-03  & 1.73  &6.021e-02 	 &1.19   \\
  32   &   1.088e-03 & 2.00&  4.061e-04  & 1.98  &2.639e-02 	 &1.05  \\
  64   &   2.717e-04 &     &  1.028e-04  &       &1.275e-02  &   \\
  %\hline
  %Slope &   1.574504 &   1.666216 &   1.467853 &   1.598312  \\
  \hline
  &\multicolumn{6}{c}{Randomized meshes} \\  
  \hline
  %$\frac{1}{h}$ & $||m-m^h||_{L^\infty}$ & Slope & $|| m-m^h||_{L^2}$ 
  %                                       & Slope &  $|| p-p^h||_{L^2}$ & Slope  \\
   8   &   6.889e-03  &0.89  &  3.335e-03  & 1.09 	& 1.795e-01 	 &1.31 \\
  16   &  3.722e-03  &1.62  &  1.566e-03  & 1.67	&7.254e-02 	 &1.12  \\
  32   &   1.212e-03  &1.85  &  4.921e-04  & 2.01	&3.348e-02 	& 1.04 \\
  64   &   3.370e-04  &      &  1.221e-04  &  		&1.630e-02  &   \\
  \hline
  & \multicolumn{6}{c}{Smoothly distorted meshes} \\  
  \hline
  %$\frac{1}{h}$   &  $|| m-m^h||_{L^\infty}$&Slope & $|| m-m^h||_{L^2}$  & Slope &  $|| p-p^h||_{L^2}$  & Slope\\
  %\hline
  8   &  6.451e-03 & 0.57&   3.562e-03   & 1.20  & 2.522e-01 	& 1.52 	\\
  16  &  4.349e-03  & 1.49&   1.550e-03    & 1.52   &  8.808e-02 &1.45 	\\
  32  &  1.551e-03  &1.61&   5.414e-04   & 1.91  &  3.220e-02 	& 1.21	\\
  64  &   5.073e-04  &       &1.440e-04     &     &  1.393e-02 	&	\\
  \hline
  &\multicolumn{6}{c}{Polygonal meshes}  \\  
  \hline
 8  &	 7.071e-03  &	 0.72  &	 4.029e-03  &	 1.42  &	 2.301e-01  &	 1.32 \\ 
 16  &	 4.288e-03  &	 1.40  &	 1.506e-03  &	 1.60  &	 9.242e-02  &	 1.28 \\ 
 32  &	 1.623e-03  &	 1.71  &	 4.970e-04  &	 1.81  &	 3.794e-02  &	 1.10 \\ 
 64  &	 4.957e-04  &	 1.85  &	 1.422e-04  &	 2.28  &	 1.765e-02  &	 1.06 \\ 
 128  &	 1.372e-04  &	       &	 2.929e-05  &	       &	 8.495e-03  &	       \\
  \hline
  &\multicolumn{6}{c}{Logically square meshes in the circular domain}  \\  
  \hline
  %$\frac{1}{h}$   &  $|| m-m^h||_{L^\infty}$&Slope & $|| m-m^h||_{L^2}$  & Slope & $|| p-p^h||_{L^2}$  & Slope\\
  %\hline
  8  &  2.388e-02  & 1.68 &  3.689e-03  & 1.66	&  1.293e-01 	& 1.35  \\
  16 &  7.451e-03  & 1.87 &  1.166e-03  & 1.90  &  5.086e-02 	& 1.11 	\\
  32 &  2.032e-03  & 1.95 &  3.120e-04  & 1.99	&  2.354e-02 	& 1.02 	\\
  64 &  5.268e-04  &      &  7.856e-05  &	&  1.159e-02	& \\
\end{tabular}
\caption{Convergence analysis of Algorithm 1 for the case of Dirichlet boundary conditions.}
\label{table:dirichlet}
\end{table}

\begin{figure}[h]
  \includegraphics[width=.495\linewidth,trim=5 0 40 20,clip]{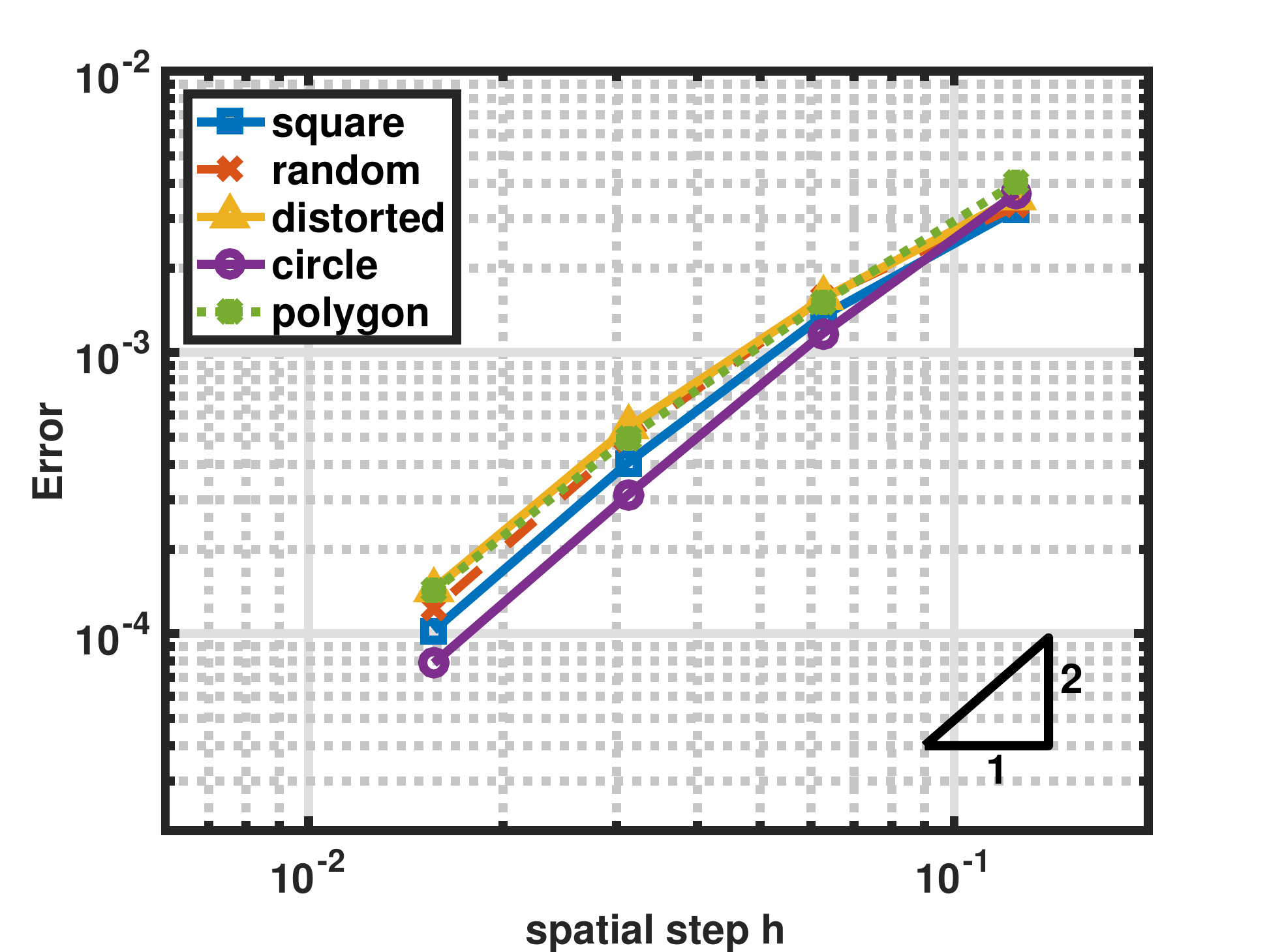}
  \includegraphics[width=.495\linewidth,trim=5 0 40 20,clip]{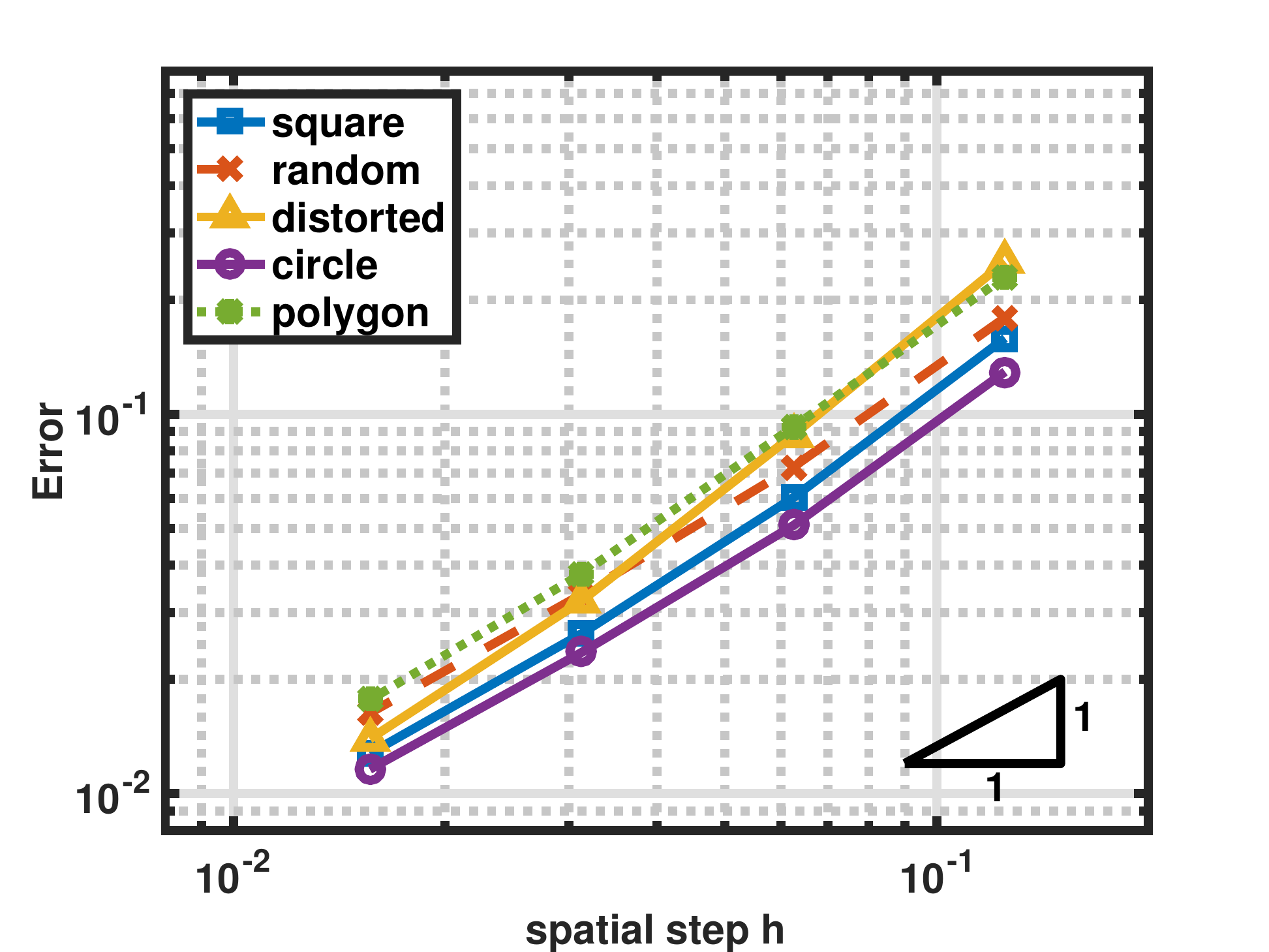}
  \caption{Error plot of $\| \mb{m}^h - \mb{m}^I \|_{\cal Q}$ (left) and 
  $\| \mb{p}^h - \mb{p}^I \|_{\cal F}$ (right) with respect to the mesh size $h$ 
  for the problem with the Dirichlet boundary conditions.}
  \label{figure:dirichlet}
\end{figure}

%\begin{figure}[h]
%  \includegraphics[width=.495\linewidth,trim=5 0 40 20,clip]{poly}
%  \caption{Error plot of $\| \mb{m}^h - \mb{m}^I \|_{\cal Q}$  and 
%  $\| \mb{p}^h - \mb{p}^I \|_{\cal F}$ with respect to the mesh size $h$ 
%  for general polygonal domain (Fig.~\ref{figure:meshpoly}) }
%  \label{figure:poly}
%\end{figure}

%%%%%%%%%%%%%%%%%%%%%%%%%%%%%%%%%%%%%%%%%%%%%%%%%%%%%%%%%%%%%%%%%%%%%
\clearpage
\subsection{NIST micromag standard problem 4}
\label{subsec:nist}

The micromag standard problem 4 simulates the magnetization dynamics in 
a permalloy thin film with two different applied fields \cite{nist4}. 
The thin film has dimensions $500 \text{nm} \times 125\text{nm} \times 3\text{nm}$.
Before its nondimensionalization, the LL equation \cite{miltat2007numerical} reads
\begin{equation}\label{eq:LLnist}
  \Frac{\partial \mb{M}}{\partial t'} 
  = -\frac{\gamma}{1+\alpha^2} \mb{M} \times \mb{H} 
    -\frac{ \gamma \alpha}{M_s(1+\alpha^2)} \mb{M} \times (\mb{M} \times \mb{H})
\end{equation}
subject to the homogeneous Neumann boundary condition on $\partial \Omega$ and 
the initial condition described below.
Here
\begin{equation}\label{eq:num:H}
  \mb{H} = -\frac{1}{\mu_0 M_s} \der{\mathcal{E}}{\mb{m}},
  \quad
  \mathcal{E}(\mb{M}) = \int_\Omega \frac{A}{2M_s^2} |\nabla \mb{M}|^2 
                      - \mu_0(\mb{H}_{e} \cdot \mb{M}) 
                      - \frac12 \mu_0 (\mb{H}_s \cdot \mb{M}) \dx,
\end{equation}
where $\mb{M}=M_s \mb{m}$,
$\mb{H}_e$ is the external field and $\mb{H}_s$ is the stray field.  
Other parameters are the exchange constant $A = 2.6 \times 10^{-11}\, [J \cdot m^{-1}$], 
saturation magnetization $M_s= 8 \times 10^5\, [A \cdot m^{-1}]$, 
gyromagnetic ratio $\gamma = 2.21\times 10^{5}\, [m \cdot A^{-1} \cdot s^{-1}]$,  
magnetic permeability of vacuum $\mu_0 = 4 \pi \times 10^{-7}\, [N \cdot A^{-2}]$ and 
the dimensionless damping parameter $\alpha = 0.02$.  

By rescaling $\mb{H}= M_s \mb{h}$, $\mb{H}_s = M_s \mb{h}_s$, $\mb{H}_e=M_s \mb{h}_e$,
$x=Lx'$ with $L=10^{-9}$, and $t= \frac{1+\alpha^2}{\gamma M_s} t'$, we get equation
(\ref{eq:LL}) with $\eta =\frac{A}{\mu_0 M_s^2 L^2}$ and equation (\ref{eq:LLenergy}) with $Q=0$.
The initial state is an equilibrium S-state as in Fig.~\ref{figure:S} which is obtained by 
applying an external field of $2 \,T$ along direction $[1,1,1]$ and slowly reducing it to 
zero by $0.02 \,T$ each time step \cite{nist4,kritsikis2014beyond}. 

For a thin film, we may assume that the magnetization is constant in the vertical direction
and solve the two-dimensional LL equation.
Let us consider a $100 \times 25$ rectangular mesh of square cells with size $h_x = h_y = 5 \text{nm}$. 
We use the explicit time integration scheme with time step  
$\hat k=\frac{0.005}{\gamma M_s}\approx 28.28 \text{ fs}$ and implicit time discretization scheme
with five different time steps  $\hat k=\frac{0.01}{\gamma M_s}\approx 56.56 \text{ fs}$, 
$\hat k\approx 0.14 \text{ ps}$, $\hat k\approx 0.28 \text{ ps}$, $\hat k\approx 0.57 \text{ ps}$, and $\hat k\approx 1.13 \text{ ps}$.
\begin{figure}[h]
\centering
  \includegraphics[width=.62\linewidth,trim=20 200 15 10,clip]{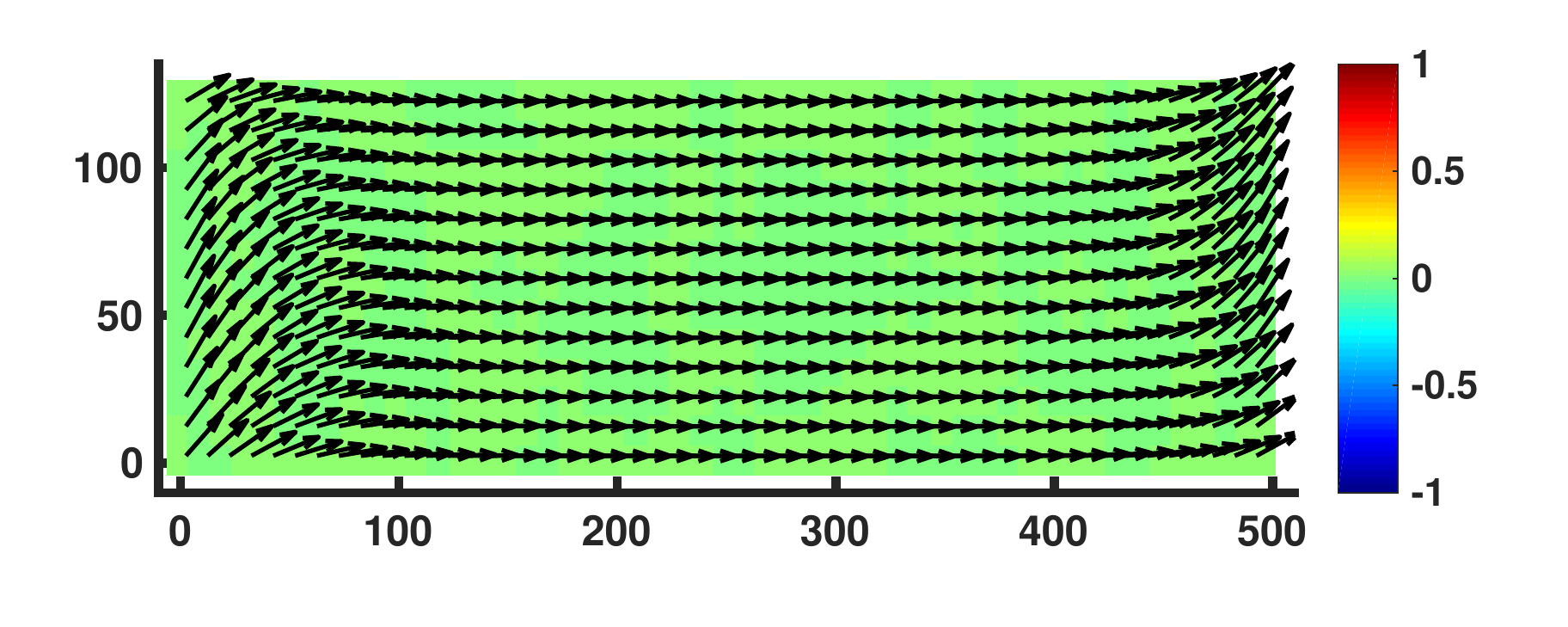}
\caption{NIST 4th problem: the initial equilibrium S-state. The vectors in the plot denote the $m_x$ and $m_y$ components, and the color denotes the $m_z$ component.}
\label{figure:S}
\end{figure}
\begin{figure}[h]
  \includegraphics[width=.495\linewidth,trim=15 5 50 10,clip]{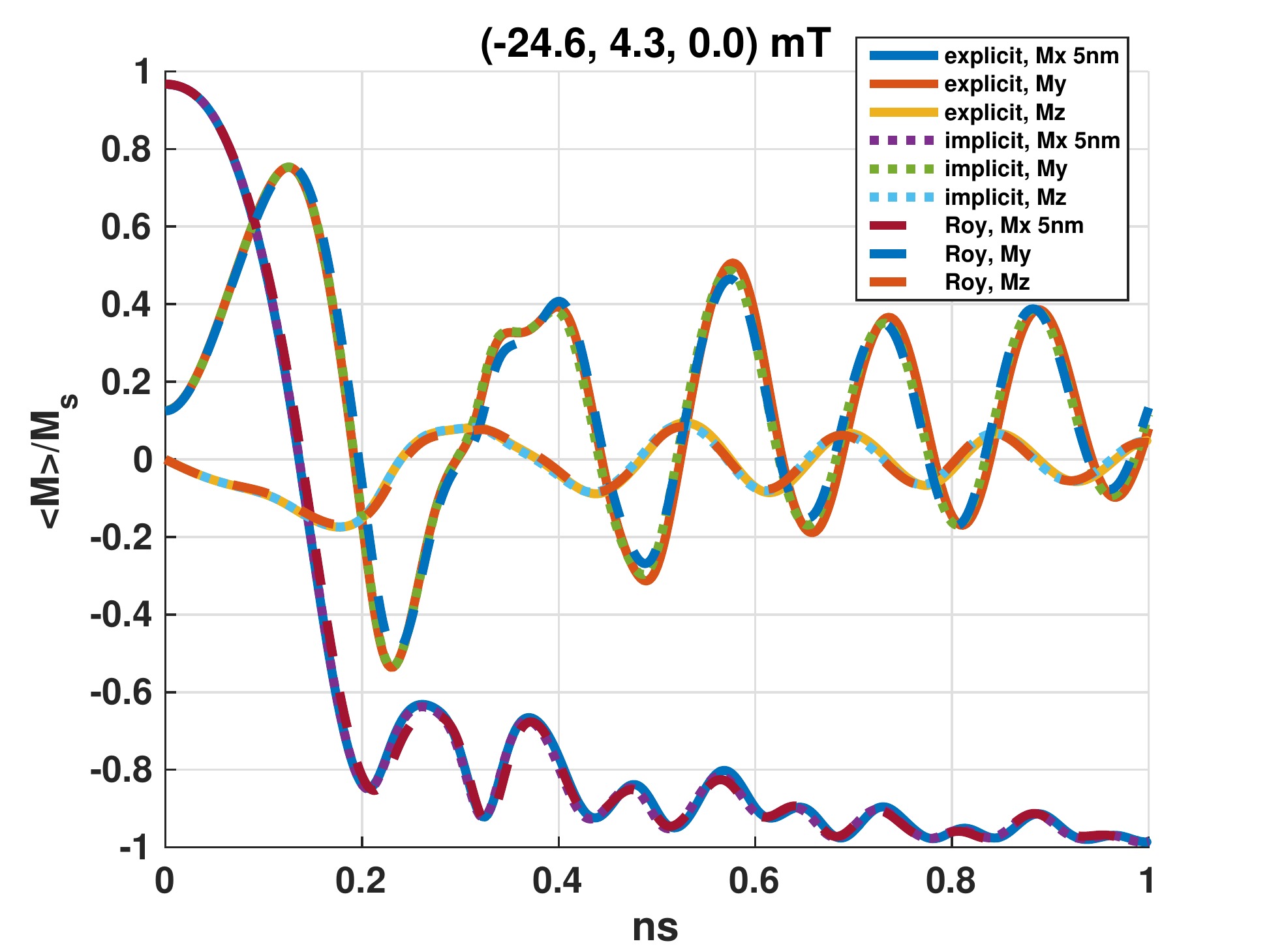}
  \includegraphics[width=.495\linewidth,trim=20 5 50 10,clip]{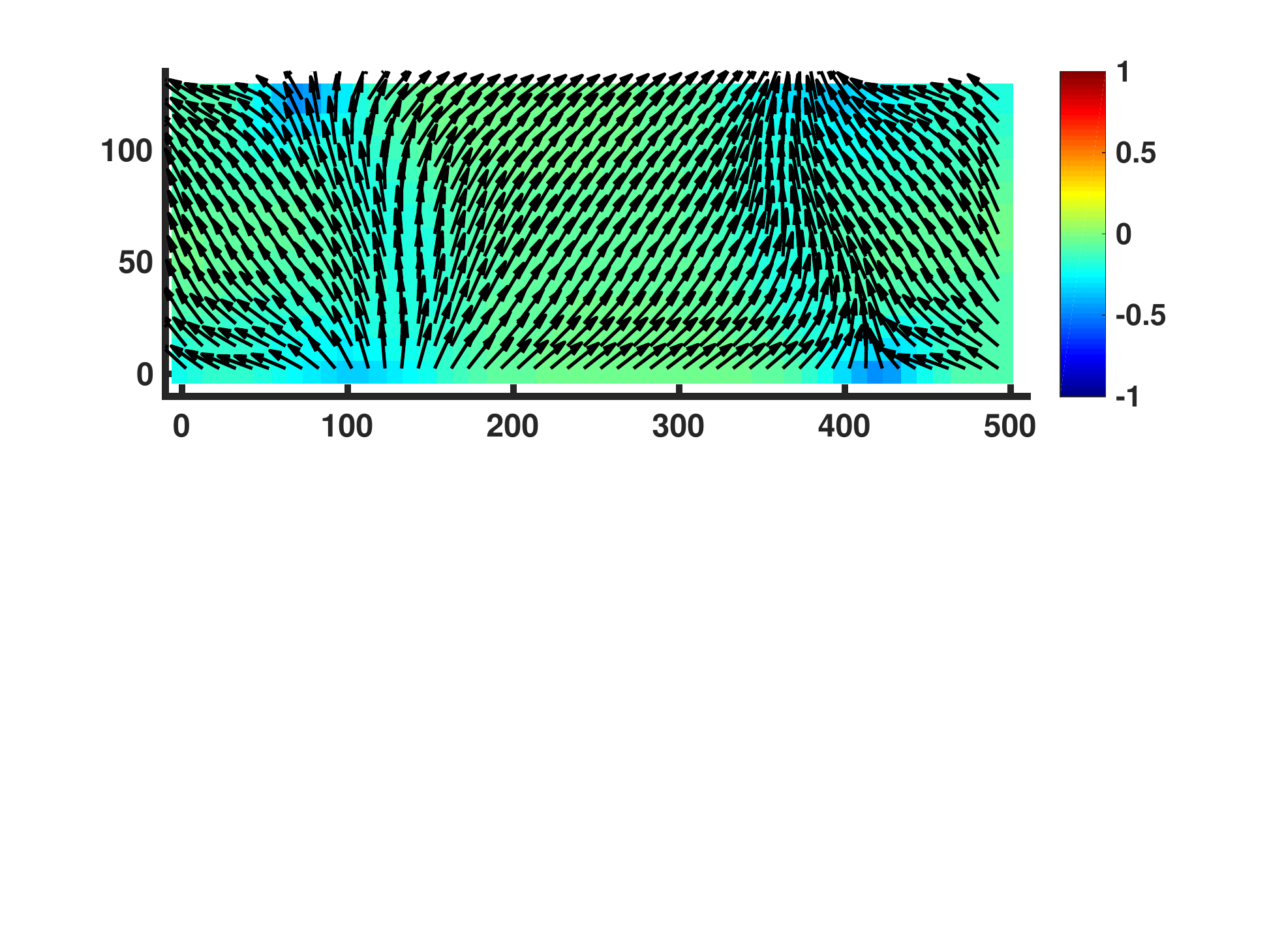}
\caption{NIST 4th problem with the external field $\mu_0\,\mb{H}_e = (-24.5,\, 4.3,\, 0.0)\, [mT]$. 
Left panel compares the time evolution of the average magnetization calculated using the explicit and implicit mimetic schemes 
with Roy and Svedlindh's results in \cite{nist4}. Right panel shows the magnetization field when $\left<m_x\right>$ first crosses zero. The vectors in the plot denote the $m_x$ and $m_y$ components, and the color denotes the $m_z$ component.}
\label{figure:field1}
\end{figure}

\begin{figure}[h]
  \includegraphics[width=.495\linewidth,trim=15 5 50 10,clip]{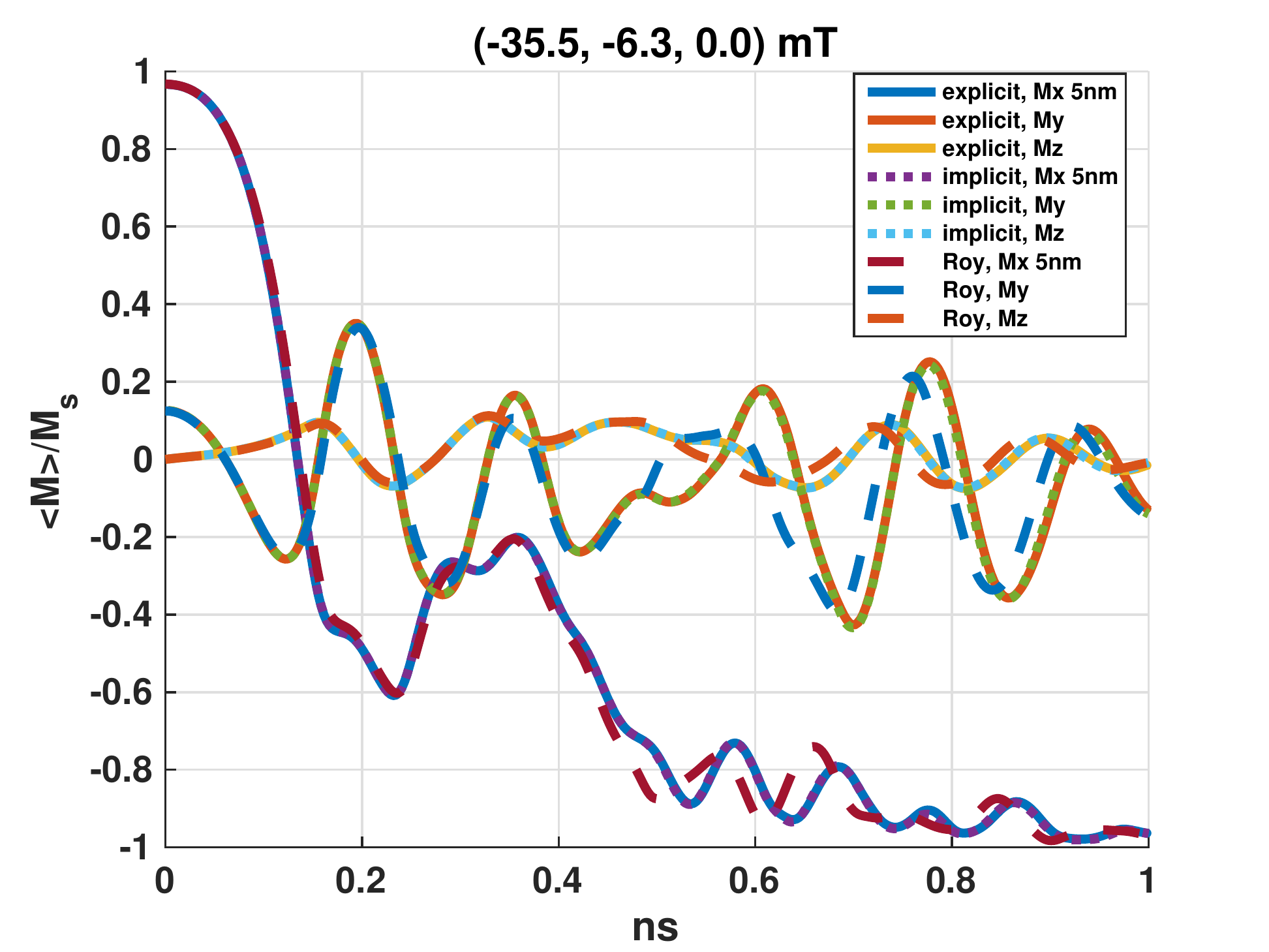}
  \includegraphics[width=.495\linewidth,trim=20 5 50 10,clip]{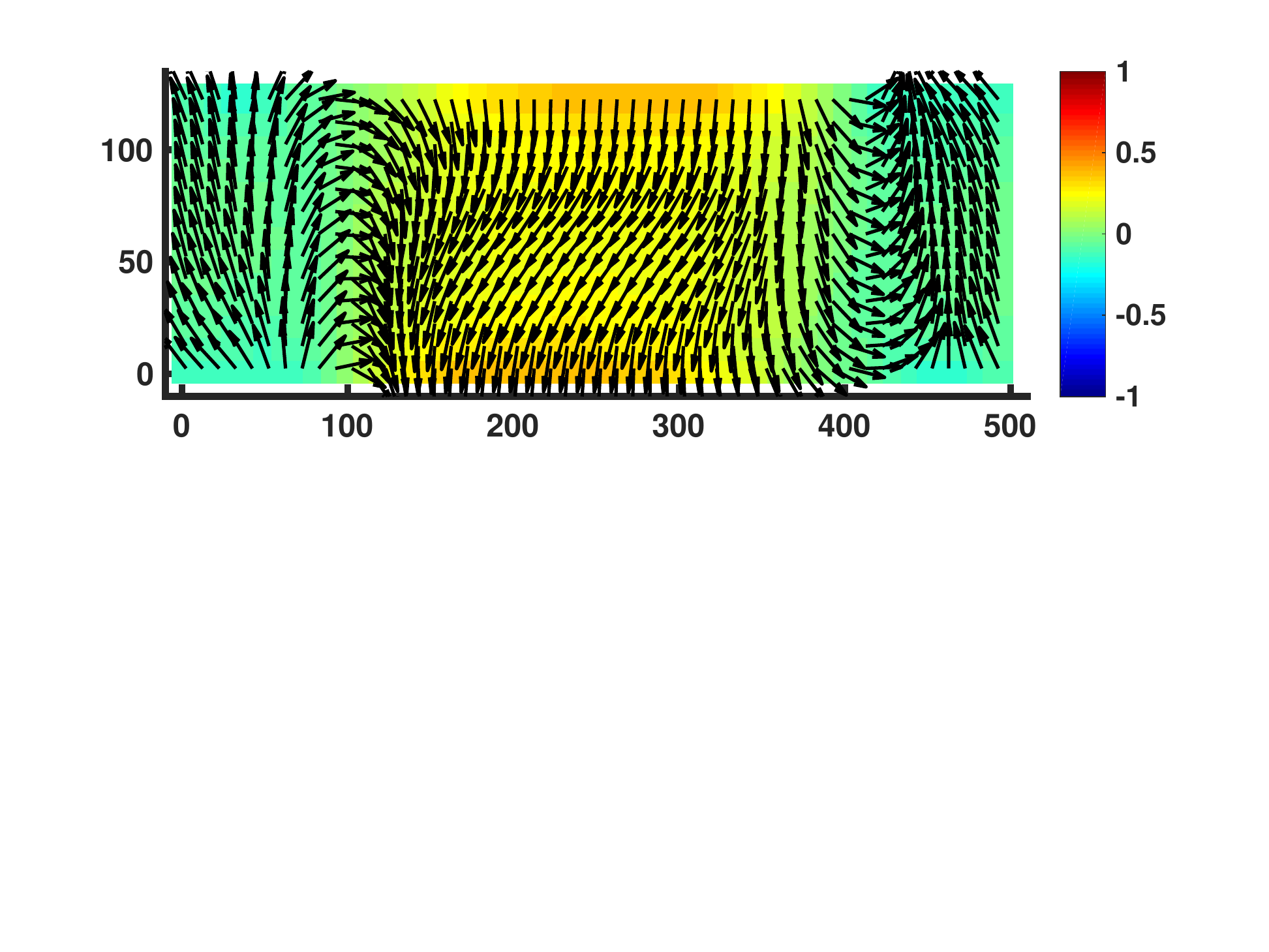}
  \caption{NIST 4th problem  with the external field $\mu_0\,\mb{H}_e = (-35.5,\, 6.3,\, 0.0)\,[mT]$.
Left panel compares the time evolution of the average magnetization calculated using the explicit and implicit mimetic schemes
with Roy and Svedlindh's results in \cite{nist4}. Right panel shows the magnetization field when $\left<m_x\right>$ first crosses zero. The vectors in the plot denote the $m_x$ and $m_y$ components, and the color denotes the $m_z$ component.}
%\caption{NIST 4th problem with the external field $\mu_0\,\mb{H}_e = (-35.5,\, 6.3,\, 0.0)\,[mT]$.}
\label{figure:field2}
\end{figure}

The time evolution of the magnetization is simulated 
using two different applied fields. 
The first field is $\mu_0 \mb{H}_e= (-24.6,\, 4.3,\, 0.0)\, [mT]$ which makes angle of
approximately $170$ degrees with the positive direction of the x-axis. 
The second field is $\mu_0 \mb{H}_e= (-35.5,\, 6.3,\, 0.0)\, [mT]$ which makes angles of
approximately $190$ degrees with the positive direction of the x-axis.

The evolution of the average magnetization 
$$
 \left<\mb{m}\right>= \frac{1}{N_E} \sum _{E\in \Omega_h} \mb{m}_E
$$ 
with the first field is shown in 
Fig.~\ref{figure:field1} and compared with the results obtained by Roy and Svedlindh in \cite{nist4}. 
They used a finite difference method (leading to the conventional 5-point approximation for the Laplacian)
and RK4 for the time stepping with time step $\hat k\approx11\text{ fs}$. 
Also, the magnetization field when $\left<m_x\right>$ first crosses zero
is shown on the right panel in Fig.~\ref{figure:field1}.  
The evolution of the magnetization is qualitatively in a very good agreement.

Consider now the second external field.
The evolution of the average magnetization and the magnetization field when $\left<m_x\right>$ first 
crosses zero is shown in Fig.~\ref{figure:field2}. 
As mentioned in \cite{nist4}, solutions obtained with different schemes begin to diverge approximately after $0.35\, ns$.
Note that we have a qualitatively good agreement until this time moment.

Furthermore, in Fig.~\ref{figure:nistk}, the evolution of the magnetization for both applied fields
with different time steps is plotted, which shows the stability and temporal convergence of the implicit scheme of Algorithm~1.

\begin{figure}[h]
  \includegraphics[width=.495\linewidth,trim=10 5 30 10,clip]{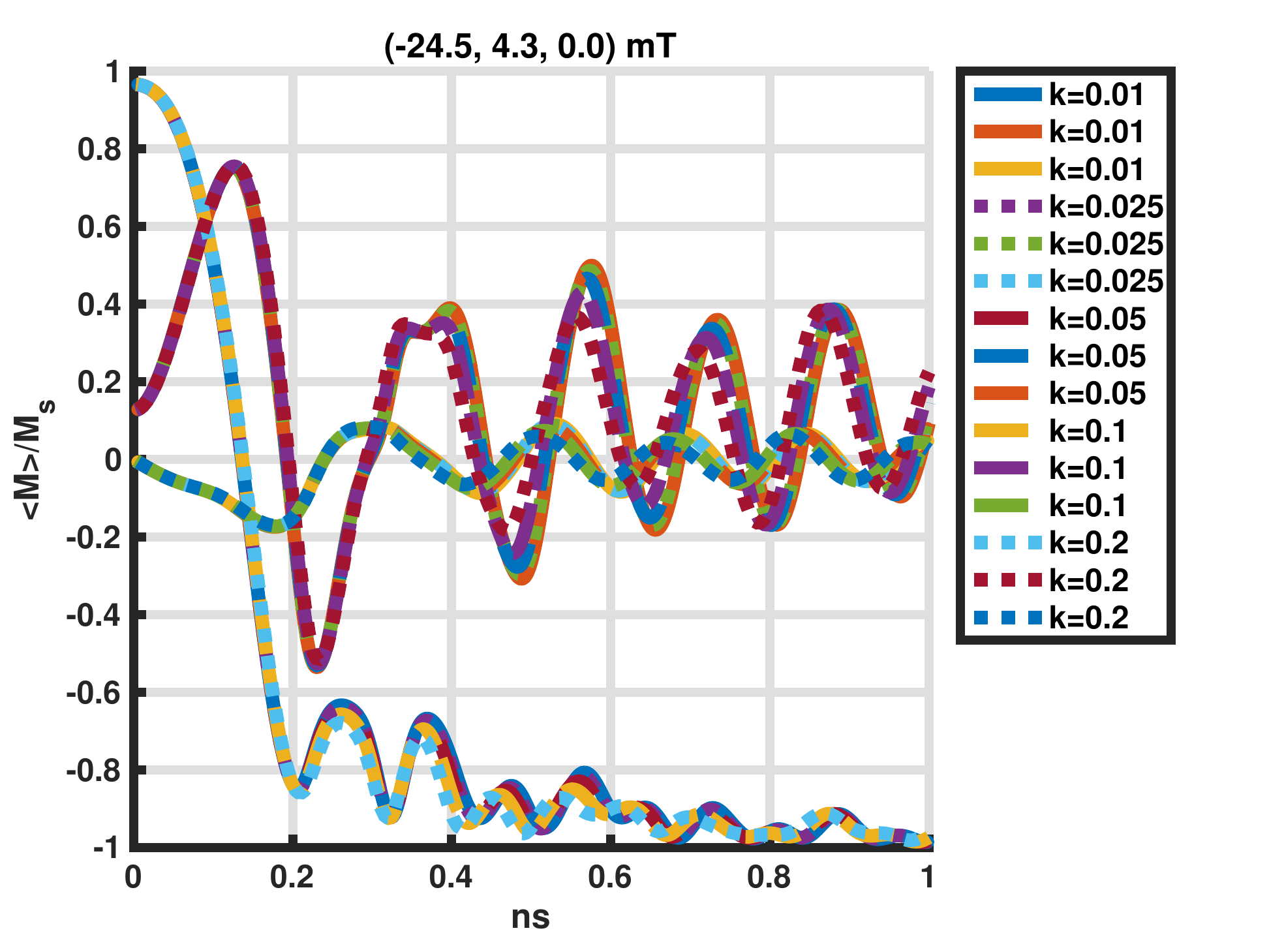}
  \includegraphics[width=.495\linewidth,trim=10 5 30 10,clip]{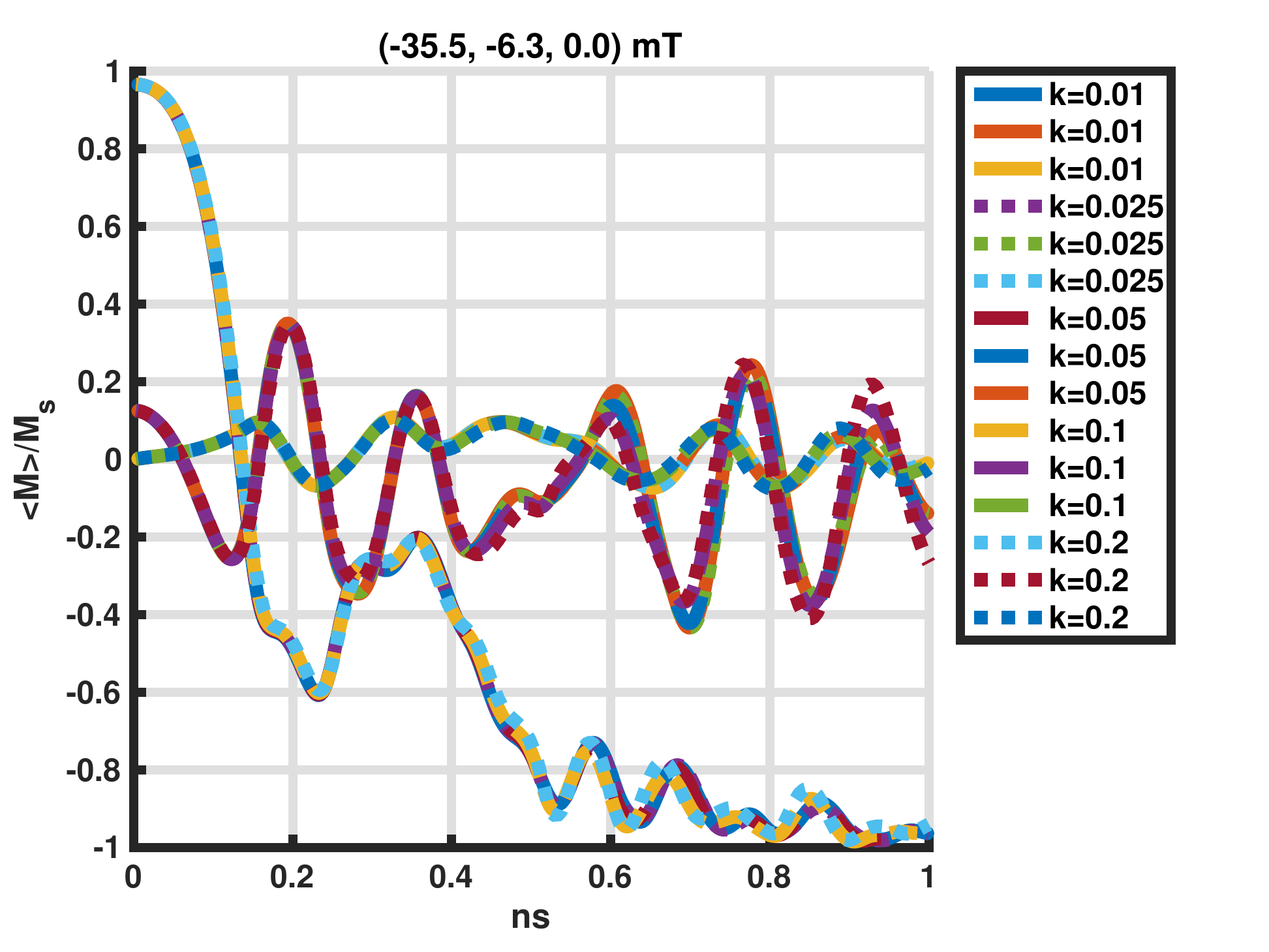}
\caption{NIST 4th problem: evolution of average magnetization computed using Algorithm 1 with 
various time steps $\hat k=\frac{k}{\gamma M_s}$  for applied fields $\mu_0\mb{H}_e = (-24.5,\, 4.3,\, 0.0)\,[mT]$ 
(left) and $\mu_0\mb{H}_e =(-35.5,\, 6.3,\, 0.0)\,[mT]$ (right).}
\label{figure:nistk}
\end{figure}

%%%%%%%%%%%%%%%%%%%%%%%%%%%%%%%%%%%%%%%%%%%%%%%%%%%%%%%%%%%%%%%%%%%%%%%%%%%%%%%%%%%%
\subsection{Domain wall structures in a thin film}
\label{subsec:domainwall}

In this subsection, we conduct numerical experiments of the domain wall structures 
in thin films with no external field using both explicit and implicit time discretization schemes.  
A similar numerical experiment was conducted in \cite{wang2006simulations} using the Gauss-Seidel 
projection method and gradually increasing the thickness of the film. 
Before its nondimensionalization, the LL equation  \cite{wang2006simulations} reads
\begin{equation}\label{eq:LLneel}
  \Frac{\partial \mb{M}}{\partial t'} 
  = -\gamma \mu_0 \mb{M} \times \mb{H} - \frac{\gamma \alpha\mu_0}{M_s} \mb{M} \times (\mb{M} \times \mb{H})
\end{equation}
subject to homogeneous Neumann boundary conditions and the initial condition described below.
The effective field $\mb{H}$ is given by \eqref{eq:num:H} but
most of the parameters described in that subsection have different values for this experiment.
We set the exchange constant $A = 2.1 \times 10^{-11}\, [J\cdot m^{-1}]$, 
saturation magnetization $M_s = 1.71 \times 10^6\, [A \cdot m^{-1}]$, 
gyromagnetic ratio $\gamma = 1.76 \times 10^{11}\, [T^{-1}\cdot s^{-1}]$, 
magnetic permeability of vacuum $\mu_0 = 4 \pi \times 10^{-7}\, [N \cdot A^{-2}]$, 
and the dimensionless damping parameter $\alpha = 0.02$. 

Using a slightly different rescaling than above, $\mb{H}= M_s \mb{h}$, $\mb{H}_s = M_s \mb{h}_s$, 
$\mb{H}_e = M_s \mb{h}_e$, $x=Lx'$ with $L= 10^{-9}$, and
$t= \frac{1}{\mu_0 \gamma M_s} t'$, we get equation (\ref{eq:LL}) with $\eta =\frac{A}{\mu_0 M_s^2 L^2}$
and equation (\ref{eq:LLenergy}) with $Q=0$. 
We consider a rectangular thin film of size $240 \text{nm} \times 480\text{nm} \times 7\text{nm}$.
Neglecting variation of the magnetization in the vertical direction, we 
solve the two-dimensional LL equation on a $64 \times 128$ mesh of square cells
with size $h_x =h_y= 3.75\text{nm}$.
For the explicit time integration scheme, we use time step 
$k=\frac{0.01}{\mu_0 \gamma  M_s}\approx 26.44 \text{ fs}$.
For the implicit time discretization scheme, we set $k=\frac{0.25}{\mu_0 \gamma  M_s}\approx 0.66 \text{ ps}$.
The initial state is a uniform N\'eel structure, with $\mb{m}=(0,1,0)$ for $0<x<120\text{nm}$ 
and $\mb{m}=(0,-1,0)$ for $120\text{nm}<x<240\text{nm}$ as shown on the left-top panel 
in Fig.~\ref{figure:neel}.  
The other panels show evolution of the magnetization.
We observe that there is a transition from the N\'eel wall to 
four $90^{\circ}$ N\'eel walls connecting a vortex which is the equilibrium state. 
Note that we plotted the magnetizations on a coarser grid in Fig.~\ref{figure:neel} 
for better visualization.
 
\begin{figure}[h]
  \includegraphics[width=.495\linewidth,trim=150 20 100 15,clip]{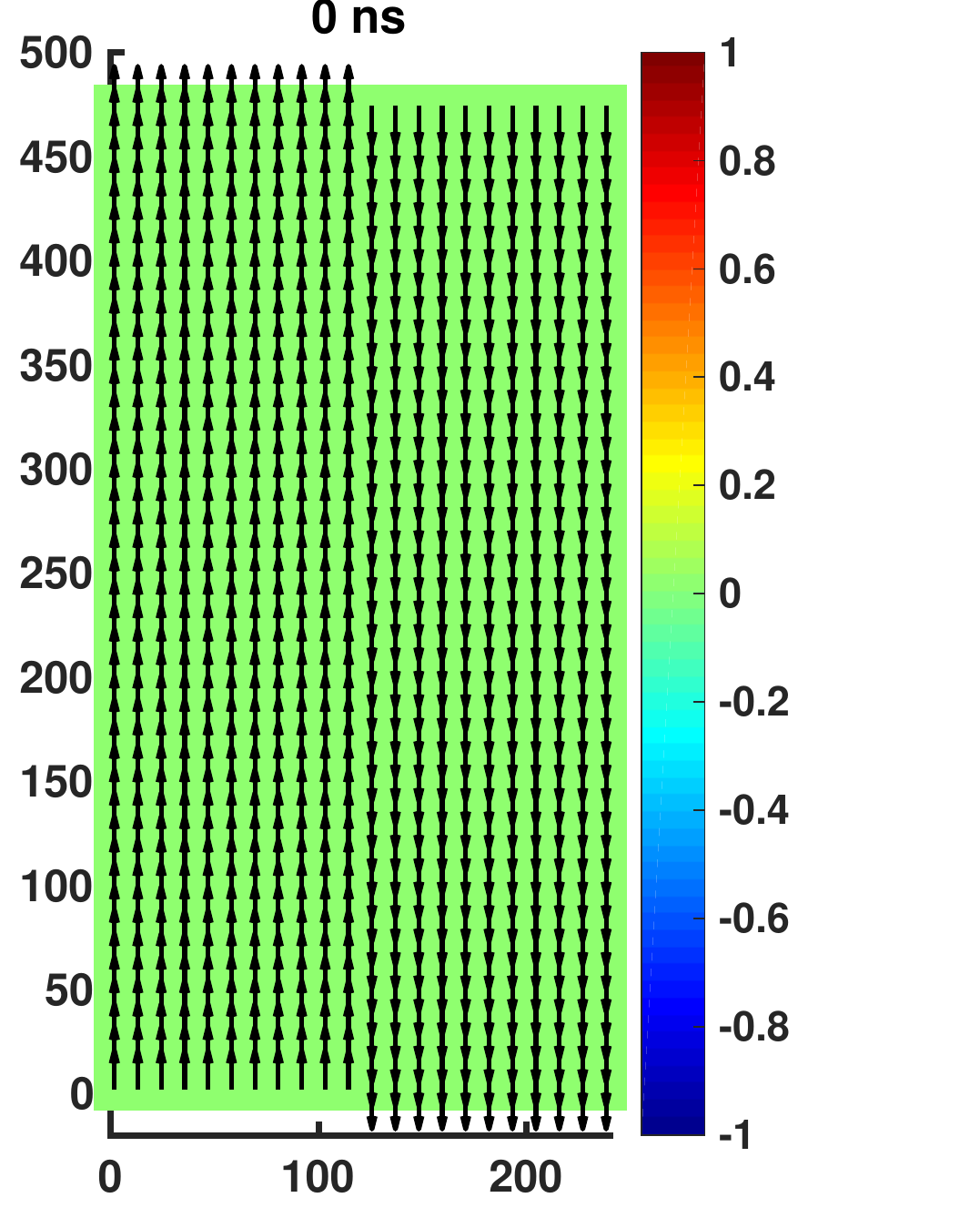}
  \includegraphics[width=.495\linewidth,trim=150 20 100 15,clip]{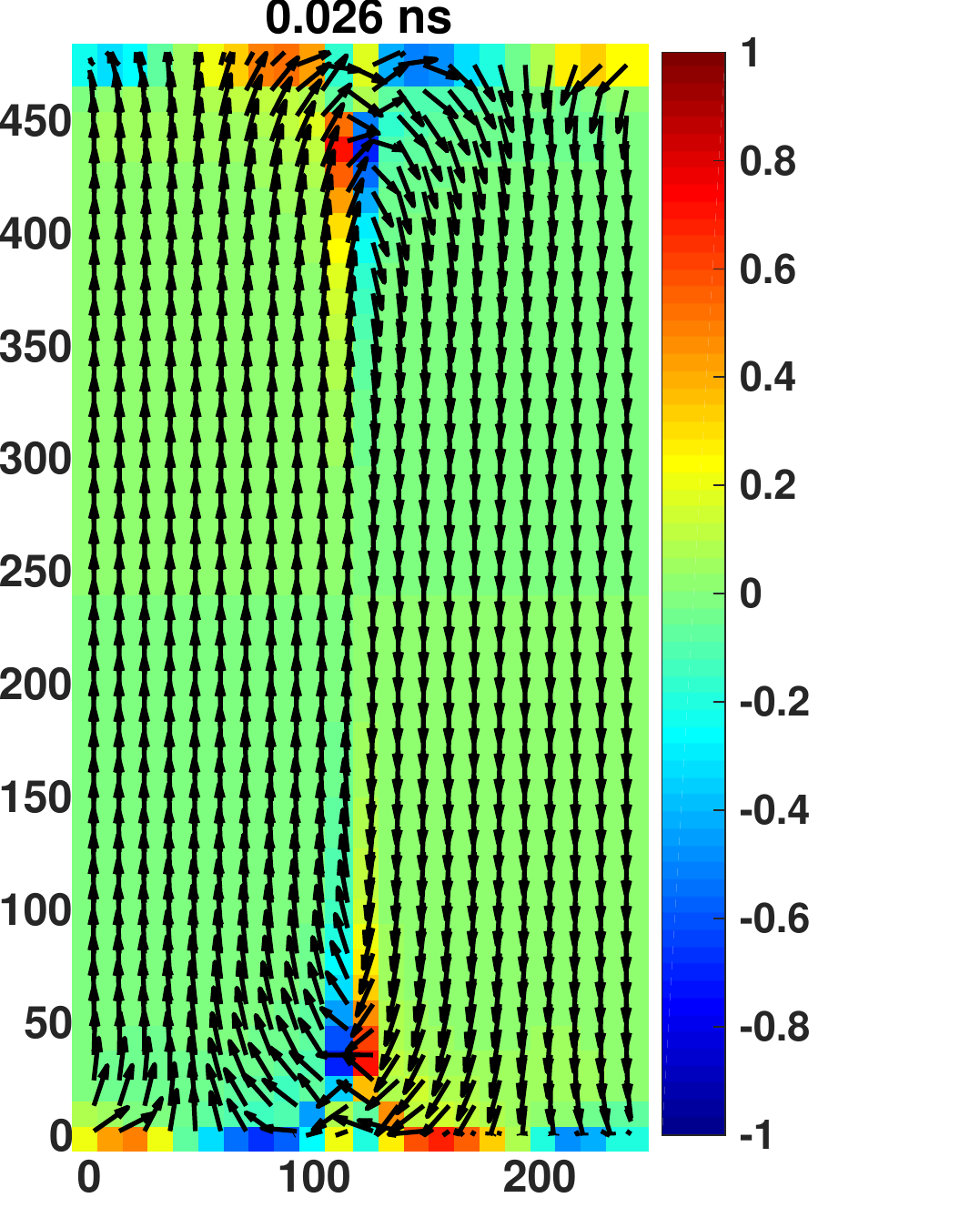}
  \includegraphics[width=.495\linewidth,trim=150 20 100 15,clip]{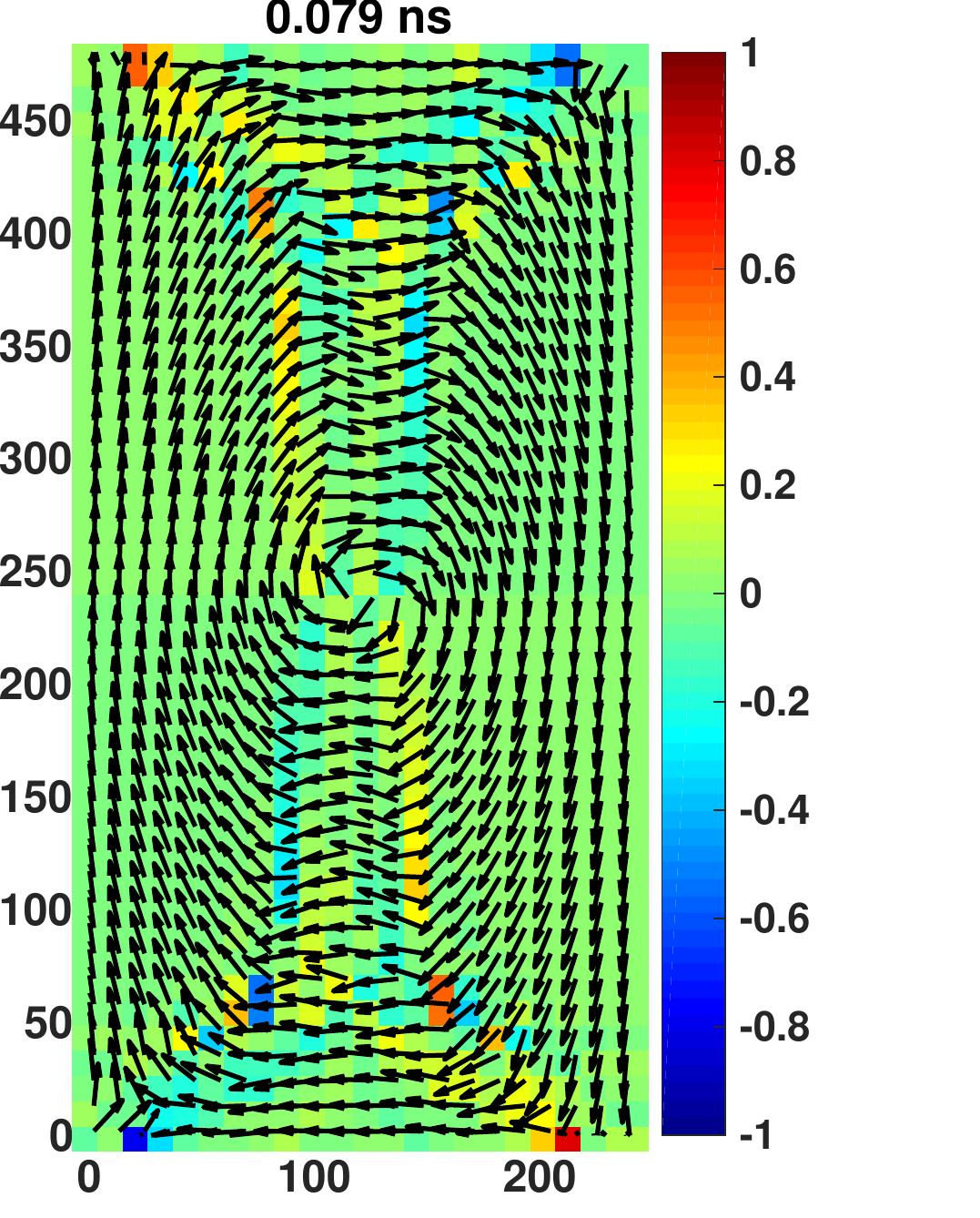}
  \includegraphics[width=.495\linewidth,trim=150 20 100 15,clip]{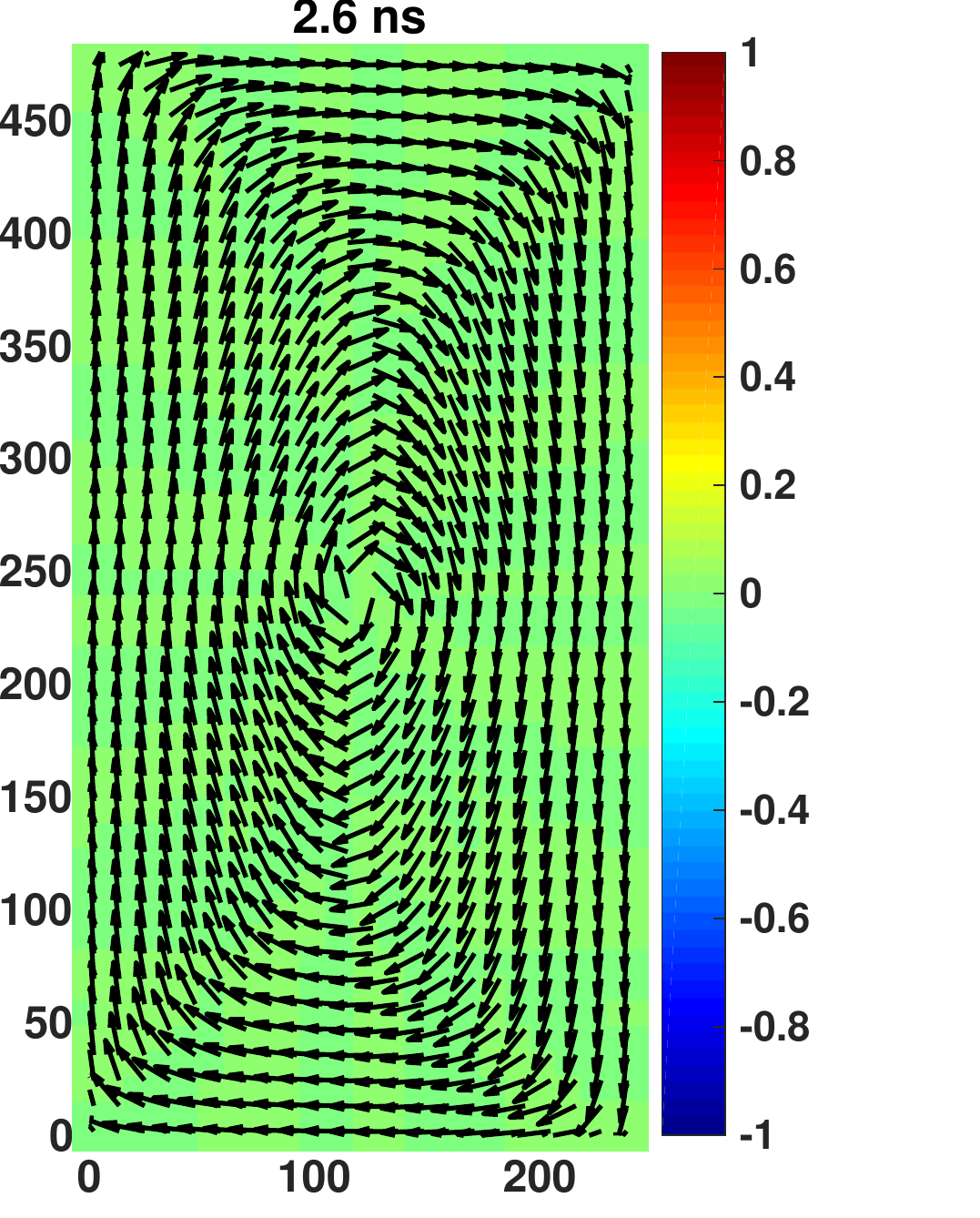}
\caption{Transition from the N\'eel wall to vortex structure. The vectors in the plot denote the $m_x$ and $m_y$ components, and the color denotes the $m_z$ component.}
\label{figure:neel}
\end{figure}

For each time step, the computational complexity involves one stray field computation 
and one linear solver accelerated with an algebraic multigrid preconditioner.
The approximate cost for the stray field computation is $O(N\log N)$ and algrebraic multigrid is $O(N)$ 
using the software library Hypre \cite{falgout2002hypre}, where $N$ is the number of degrees of freedom.
More precisely, the cost for the stray field computation in a thin film is about $8$ times the cost of a $2$D FFT.
The cost for the algebraic multigrid method is about $(\text{complexity}+1)\times \# itrs \times 216N$,
where $\# itrs$ is about $9$ and complexity is about $1.67$ in our case.

%%%%%%%%%%%%%%%%%%%%%%%%%%%%%%%%%%%%%%%%%%%%%%%%%%%%%%%%%%%%%%%%%%%%%%%%%%%%%%%%%%%%
\subsection{Adaptive mesh refinement}
\label{subsec:amr}
Dynamics of domain walls shows a strong need for adaptive meshes.
In this subsection, we compare performance of the MFD method on uniform and 
locally refined meshes with prescribed structure.
Development of a true adaptive algorithm is beyond the scope of this paper.
We consider the 1D steady-state solution with the Dirichlet boundary conditions:
\begin{equation}\label{eq:steady}
\begin{aligned}
  &m_x(x_1,x_2,t) =\sin (  \phi(x_1,x_2,t) ) \\
  &m_y(x_1,x_2,t) =\cos (  \phi(x_1,x_2,t) ) \\
  &m_z(x_1,x_2,t) = 0,
\end{aligned}
\qquad \phi(x_1,x_2,t) = \pi \left(1+e^{-s \pi (x_1-b/2)}\right)^{-1},
\end{equation}
where $b=1$ and $s=20$.
This is a steady-state solution of the LL equation (\ref{eq:LL})-(\ref{eq:LLenergy}) with 
the external field 
\begin{equation}\label{eq:hext}
\begin{aligned}
  &(\mb{h_e})_x(x_1,x_2,t) =\left(\der{\phi(x_1,x_2,t)}{x}\right)^2 \sin (  \phi(x_1,x_2,t) ) - \dder{\phi(x_1,x_2,t)}{x} \cos (  \phi(x_1,x_2,t) ) \\
  &(\mb{h_e})_y(x_1,x_2,t) =\left(\der{\phi(x_1,x_2,t)}{x}\right)^2\cos (  \phi(x_1,x_2,t) ) + \dder{\phi(x_1,x_2,t)}{x} \sin (  \phi(x_1,x_2,t) ) \\
  &(\mb{h_e})_z(x_1,x_2,t) =0, 
\end{aligned}
\end{equation}
$Q=0$ and $\mb{h}_s=0$.
This solution has a sharp transition on interval $0.4<x_1<0.6$
and is almost constant on the other regions as shown on the left-top panel in Fig~\ref{figure:amr}.
The locally refined meshes are polygonal meshes of squares and degenerate pentagons.
They are shown on the remaining panels in Fig~\ref{figure:amr}. 

\begin{figure}[h]
  \includegraphics[width=.495\linewidth,trim=5 0 40 30,clip]{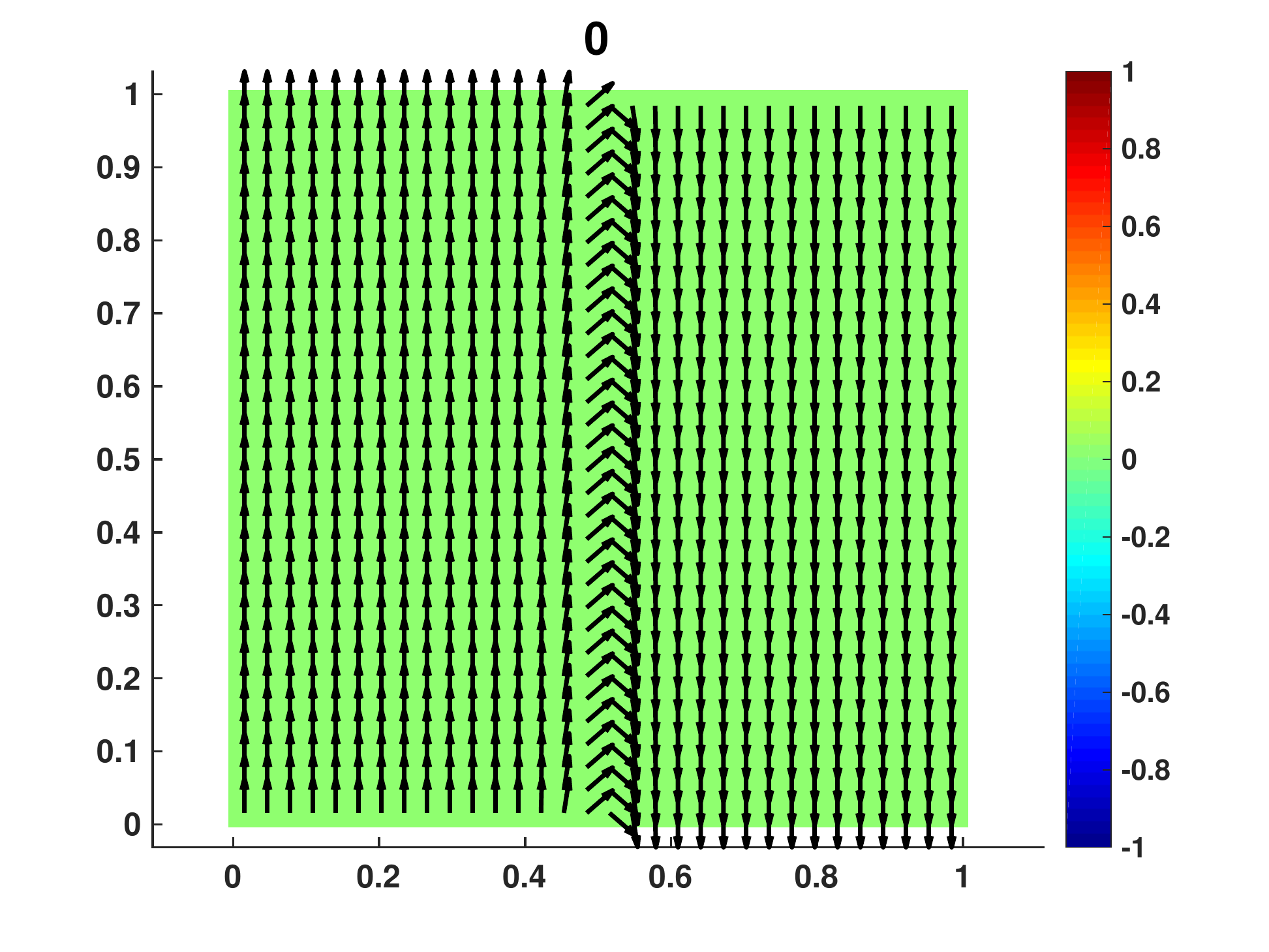}
  \includegraphics[width=.495\linewidth,trim=5 0 40 20,clip]{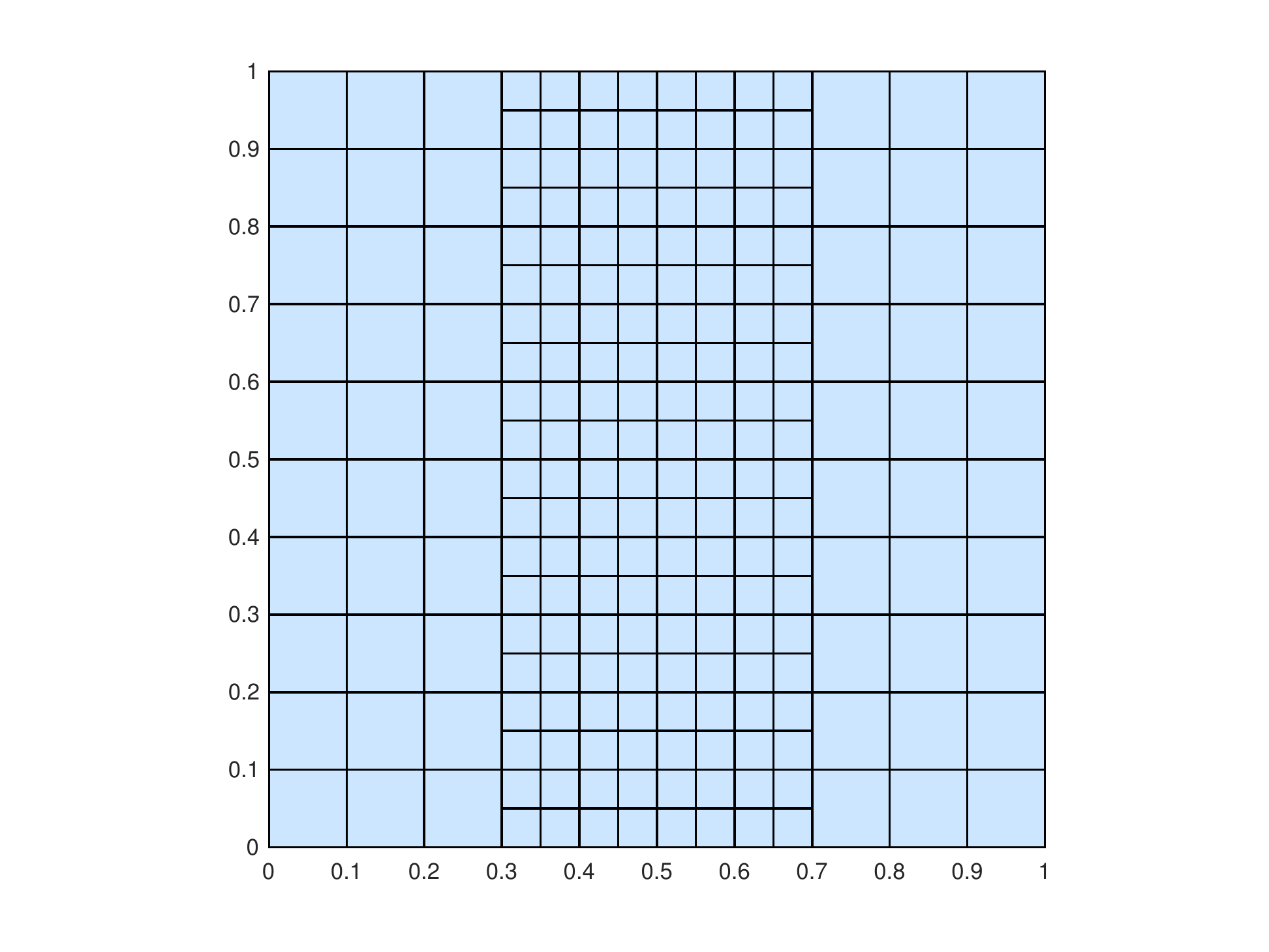}\\
  \includegraphics[width=.495\linewidth,trim=5 0 40 20,clip]{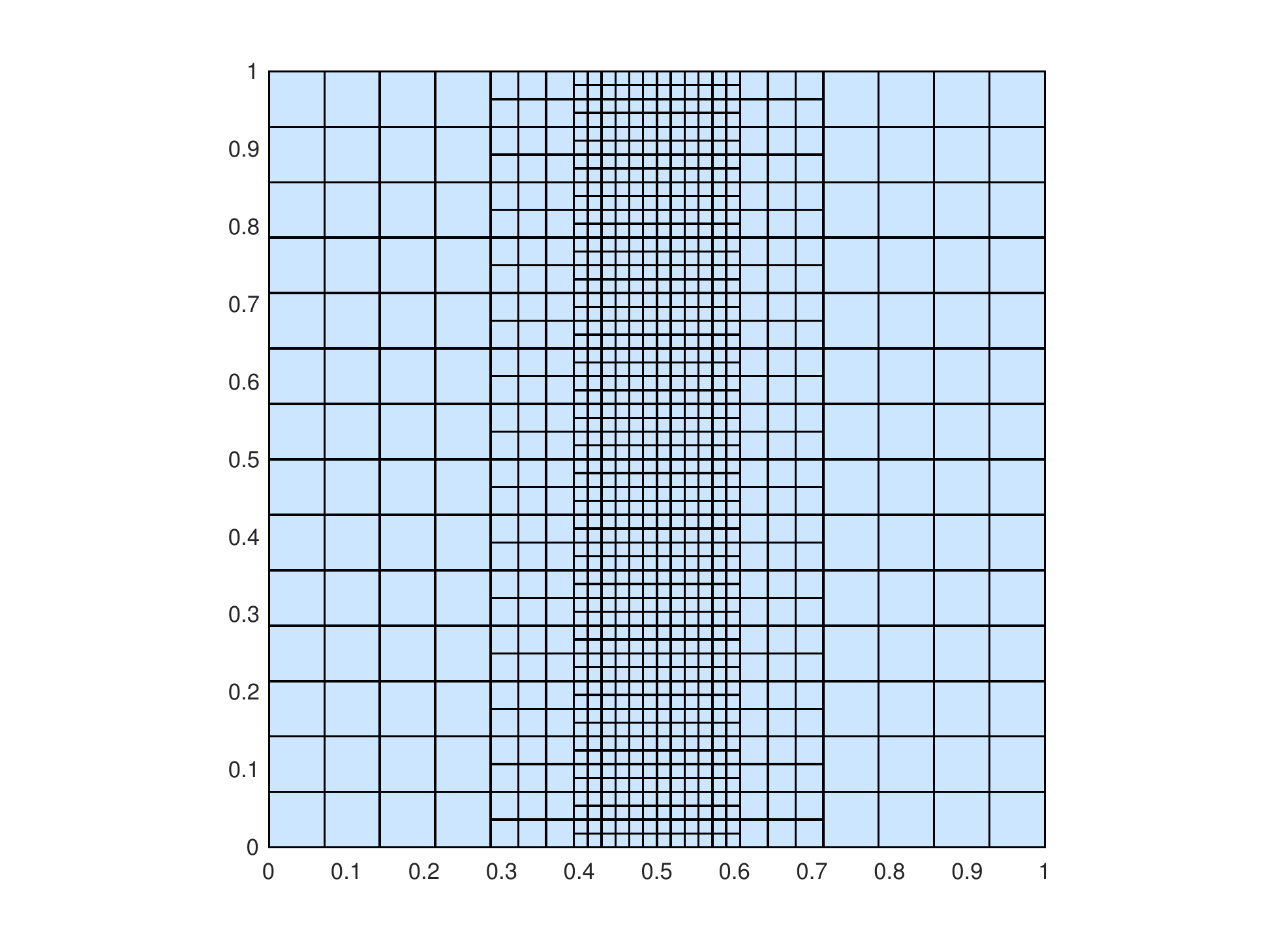}
  \includegraphics[width=.495\linewidth,trim=5 0 40 20,clip]{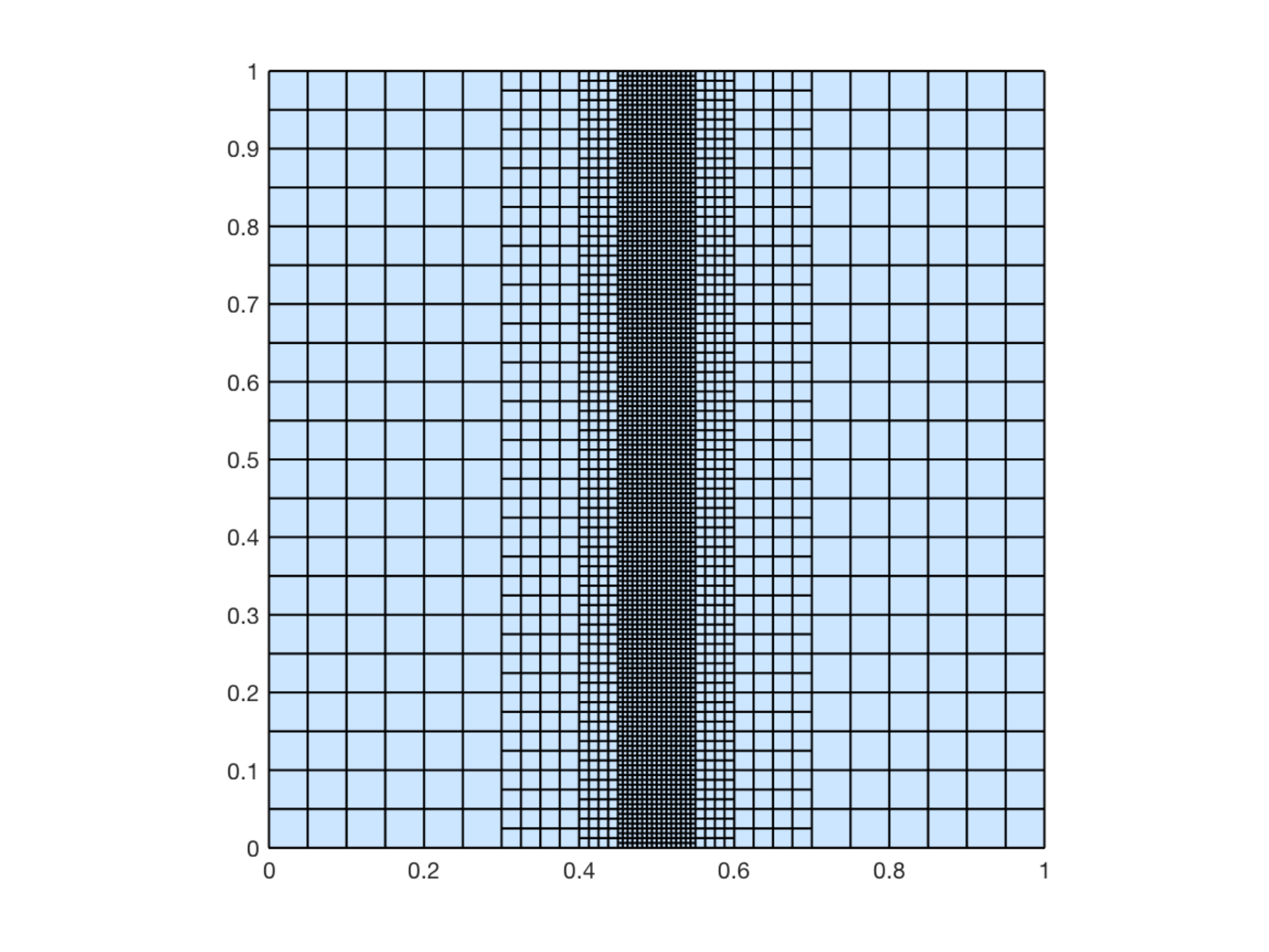}
  \caption{Steady state solution on a uniform mesh and three first locally refined meshes.}
  \label{figure:amr}
\end{figure}

The convergence results are summarized in Table~\ref{table:amr}.
For about the same numerical cost, the locally refined meshes lead to much more accurate simulations.
We expect even better behavior for adaptive meshes built using an error indicator.

\begin{table}[h]
\centering
\begin{tabular}{c|c|c|c|c}
\multicolumn{5}{c}{Uniform square meshes} \\
\hline
Number of cells &  $\| \mb{m}^h - \mb{m}^I\|_{L^\infty}$ & ratio & $\| \mb{m}^h - \mb{m}^I\|_{\cal Q}$  & ratio \\
\hline
256 &	9.170e-01 &	1.51  &	3.429e-01 &	1.67 \\
1024 &	3.231e-01 &	2.67 &	1.081e-01 &	2.81 \\
4096 &	5.072e-02 &	2.09  &	1.542e-02 &	2.10 \\
16384 &	1.192e-02 &	 &	3.605e-03 &	 \\
\hline
\multicolumn{5}{c}{Adaptive mesh} \\
\hline
220 &	8.993e-01 &	3.52 &	2.906e-01 &	3.60 \\
952 &	6.846e-02 &	3.30 &	2.076e-02 &	3.36 \\
3760 &	7.099e-03 &	1.39 &	2.062e-03 &	1.99 \\
15904 &	2.609e-03 &	 &	4.915e-04 &	 \\
%------------------------------------------------------
%256 &	9.170e-01 &	1.51 &	3.429e-01 &	1.67 \\
%1024 &	3.231e-01 &	2.67 &	1.081e-01 &	2.81 \\
%4096 &	5.072e-02 &	2.09 &	1.542e-02 &	2.10 \\
%16384 &	1.192e-02 &	2.09 &	3.605e-03 &	2.10 \\
%220 &	8.993e-01 &	3.52 &	2.906e-01 &	3.60 \\
%952 &	6.846e-02 &	3.30 &	2.076e-02 &	3.36 \\
%3760 &	7.099e-03 &	1.39 &	2.062e-03 &	1.99 \\
%15904 &	2.609e-03 &	1.39 &	4.915e-04 &	1.99 \\
%   \multicolumn{5}{c||}{Uniform mesh} & \multicolumn{5}{|c}{Adaptive Mesh}\\  
%   \hline
%\# of cells   &  $L^\infty$ error  & ratio &  $L^2$ error    & ratio &   \# of cells  &  $L^\infty$ error   & ratio    &  $L^2$ error   & ratio \\
%  \hline
%256 &	9.170e-01 &	0.75 &	3.429e-01 &	0.83 &	220 &	8.993e-01 &	1.76 &	2.906e-01 &	1.80\\
%1024 &	3.231e-01 &	1.34 &	1.081e-01 &	1.40 &	952 &	6.846e-02 &	1.65 &	2.076e-02 &	1.68\\
%4096 &	5.072e-02 &	1.04 &	1.542e-02 &	1.05 &	3760 &	7.099e-03 &	0.69 &	2.062e-03 &	0.99\\
%16384 &	1.192e-02 &	1.04 &	3.605e-03 &	1.05 &	15904 &	2.609e-03 &	0.69 &	4.915e-04 &	0.99\\
\end{tabular}
\caption{Comparison of errors between uniform and locally refined meshes, see Fig.~\ref{figure:amr}.}
\label{table:amr}
\end{table}

\section*{Acknowledgements}

This work was carried out under the auspices of the National Nuclear Security Administration
of the U.S. Department of Energy at Los Alamos National Laboratory under Contract No.
DE-AC52-06NA25396.
The first author was supported by the U.S. Department of Energy, Office of Science, 
Office of Workforce Development for Teachers and Scientists, Office of Science Graduate Student Research (SCGSR) program. 
The SCGSR program is administered by the Oak Ridge Institute for Science and Education for the DOE under contract number DE-AC05-06OR23100.
The second author acknowledges the support of the U.S. Department of Energy Office of Science
Advanced Scientific Computing Research (ASCR) Program in Applied Mathematics Research.

The authors thank Pieter Swart for useful comments that shaped our selection of numerical experiments.

\clearpage
\bibliographystyle{abbrv} %elsarticle-num}

\section*{References}
\bibliography{refs}

\end{document}